\documentclass[a4paper,11pt,twoside]{amsart}
\usepackage{amsfonts, amsmath, amssymb, appendix, mathrsfs, esint, mathabx, amsthm}
\usepackage{dsfont}
\usepackage{enumerate}
\usepackage{enumitem}
\usepackage{graphicx}
\usepackage{fullpage}
\numberwithin{equation}{section}
\usepackage{hyperref}
\usepackage{xfrac}

\usepackage{mathtools}

\makeatletter
\DeclareRobustCommand*{\bfseries}{%
  \not@math@alphabet\bfseries\mathbf
  \fontseries\bfdefault\selectfont
  \boldmath
}
\makeatother

\DeclareMathOperator*{\limin}{\underline{\lim}}
\DeclareMathOperator*{\limsu}{\overline{\lim}}

\DeclareMathOperator{\ssetminus}{\!\setminus\!}

\def\lparen{(}% left parenthesis (
\catcode`(=\active
\newcommand{(}{\ifmmode\left\lparen\else\lparen\fi}
\def\rparen{)}% right parenthesis )
\catcode`)=\active
\newcommand{)}{\ifmmode\right\rparen\else\rparen\fi}

\def\wto{\rightharpoonup}

\def\RR{\mathbb{R}}

\DeclareMathOperator{\Gammap}{\Gamma\!({\it L^p})}

\def\mesabs{{\mathfrak{A}}_{\lambda}}
\def\MM{\mathbb{M}}
\def\NN{\mathbb{N}}
\def\Mm{V}

\def\I{\mathcal I}
\def\Q{\mathrm Q}

\def\ssup{\!>\!}
\def\sinf{\!<\!}

\def\M{\mathit{m}}

\def\ZZ{\mathbb Z}
\def\D{\mathfrak{D}}

\def\diam{{\rm diam}}
\def\dom{{\rm dom}}

\def\FF{\mathcal{F}}
\def\V{\mathcal V}
\def\K{\mathcal K}

\def\eps{\varepsilon}

\def\O{\mathcal O}
\def\Ds{\overline{D}}
\def\Di{\underline{D}}

\def\Bal{\mathcal Q}

\def\S{\mathcal S}

\def\f{L}
\def\ftild{\widetilde{\f}}

\DeclareMathAlphabet{\mathpzc}{OT1}{pzc}{m}{it}

\newtheorem{theorem}{Theorem}[section]
\newtheorem{lemma}{Lemma}[section]

\newtheorem{proposition}{Proposition}[section]
\newtheorem{corollary}{Corollary}[section]
\newtheorem{definition}{Definition}[section]
\renewcommand{\qedsymbol}{\ensuremath{\blacksquare}}
\renewcommand{\liminf}{\limin}
\renewcommand{\limsup}{\limsu}

\theoremstyle{example}
\newtheorem{example}{Example}[section]
\theoremstyle{remark}
\newtheorem{remark}{Remark}[section]

\newlist{hyp}{enumerate}{1}
\setlist[hyp,1]{label={\rm (${\rm H}_{\arabic*}$)}}

\newlist{hypp}{enumerate}{1}
\setlist[hypp,1]{label={\rm (${\rm H}_{\arabic*}^\prime$)}}

\newlist{hypg}{enumerate}{1}
\setlist[hypg,1]{label={\rm (${\rm B}_{\arabic*}$)}}

\newlist{hypc}{enumerate}{1}
\setlist[hypc,1]{label={\rm (${\rm C}_{\arabic*}$)}}

\newlist{hyps}{enumerate}{1}
\setlist[hyps,1]{label={\rm (${\rm S}_{\arabic*}$)}}

\newlist{hypv}{enumerate}{1}
\setlist[hypv,1]{label={\rm (${\mathscr V}_{\arabic*}$)}}

\newlist{hypl}{enumerate}{1}
\setlist[hypl,1]{label={\rm (${\mathscr L}_{\arabic*}$)}}

\newlist{hyph}{enumerate}{1}
\setlist[hyph,1]{label={\rm (${\mathrm{H}}_{\arabic*}$)}}

\newlist{hyphp}{enumerate}{1}
\setlist[hyphp,1]{label={\rm (${\mathrm{H}}_{\arabic*}^\prime$)}}

\newlist{hypG}{enumerate}{1}
\setlist[hypG,1]{label={\rm (${\mathrm P}_{\arabic*}$)}}

\newlist{hypcprime}{enumerate}{1}
\setlist[hypcprime,1]{label={\rm (${\rm C}_{\arabic*}^\prime$)}}

\newlist{hypgr}{enumerate}{1}
\setlist[hypgr,1]{label={\rm (${\rm D}_{\arabic*}$)}}

\newlist{hyp1}{enumerate}{1}
\setlist[hyp1,1]{label={\rm ($\alph*$)}}

\newlist{hyp2}{enumerate}{1}
\setlist[hyp2,1]{label={\rm ($\roman*$)}}

\numberwithin{equation}{section}

\title[$\Gamma$-limits determined by their infima]{$\Gamma$-limits of functionals determined by their infima}

\author[Omar Anza Hafsa]{Omar Anza Hafsa}
\author[Jean Philippe Mandallena]{Jean Philippe Mandallena}

\address{UNIVERSITE de NIMES, Laboratoire MIPA, Site des Carmes, Place Gabriel P\'eri, 30021 N\^\i mes, France}
\address{Laboratoire LMGC, UMR-CNRS 5508, Place Eug\`ene Bataillon, 34095 Montpellier, France.}
\email{Omar Anza Hafsa <omar.anza-hafsa@unimes.fr>}
\email{Jean-Philippe Mandallena <jean-philippe.mandallena@unimes.fr>}
\dedicatory{In memoriam Jean-Jacques Moreau}
\begin{document}
\begin{abstract} We study the integral representation of $\Gamma$-limits of $p$-coercive integral functionals of the calculus of variations in the spirit of Dal maso and Modica (1986). We use infima of local Dirichlet problems to characterize the limit integrands. Applications to homogenization and relaxation are given.
\end{abstract}
\maketitle
\section{Introduction}
Let $m,d\ge 1$ be two integers. Let $\Omega\subset\RR^d$ be a nonempty bounded open set with Lipschitz boundary. Let $\O(\Omega)$ be the class of all open subsets of $\Omega$. We consider a family of functionals $\FF:=\{F_\eps\}_{\eps\in ]0,1]}$ with $F_\eps: W^{1,p}(\Omega;\RR^m)\times\O(\Omega)\to [0,\infty]$. We set conditions in order that each functional of the family $\FF$ can be considered as a $p$-coercive integral functional of the calculus of variations (see the ``global" conditions~\ref{C0},~\ref{C1},~\ref{C2} in Sect.~\ref{s1}). We are interested in the integral representation of the $\Gammap$-limit of $\FF$. This is an important problem in the field of $\Gamma$-convergence theory (see for instance \cite{degiorgi79}). 

Our goal is to study the conditions of the integral representation of $\Gammap$-limit by using the infima of local Dirichlet problems associated to $\FF$ as in  \cite{dalmaso-modica86, bouchitte-fonseca-mascarenhas98, BFLM02}. More precisely, we consider the behavior of
\begin{align*}
\M_\eps(u;O):=\inf\left\{F_\eps(v;O):v\in u+W^{1,p}_0(O;\RR^m)\right\}
\end{align*}
in order to find the conditions for the integral representation (see also \cite{dalmaso-modica86-2, dalmaso-modica86-3, modica86}). We propose three ``local" conditions (see~\ref{H0},~\ref{H1} and~\ref{H2} in Sect.~\ref{s1}) related to the local behavior of $\M_\eps$ which allows to prove $\Gammap$-convergence of the family $\FF(\cdot;O)$ with integral representation of the $\Gammap$-limit $\FF_0(\cdot;O)$
\begin{align*}
\FF_0(u;O)=\int_O \f_0(x,u(x),\nabla u(x))dx
\end{align*}
where $u\in {\rm M}_\FF(O)$ (see Definition~\eqref{Msets} for ${\rm M}_\FF(O)$) and 
\begin{align*}
\f_0(x,u(x),\nabla u(x))=\limsup_{\rho\to 0}\limsup_{\eps\to 0}\frac{\M_\eps(u_x;\Q_\rho(x))}{\rho^d}\quad\mbox{ with }u_x(\cdot)=u(x)+\nabla u(x)(\cdot-x).
\end{align*}
The main difficulty is to obtain an upper bound under integral form for the $\Gammap$-$\limsup$. More precisely, we show, in Sect.~\ref{Vitali-envelope} together with Sect.~\ref{Proof-main-theorem}, that the Vitali envelope (which is an envelope of Carath\'eodory type where the arbitrary coverings are replaced by Vitali coverings) $V_+(u;\cdot)$ of the set function $\O(\Omega)\ni V\mapsto \limsup_{\eps\to0}\M_\eps(u;V)$ when $u\in {\rm M}_\FF(O)$ satisfies 
\begin{align*}
\Gammap\mbox{-}\limsup_{\eps\to 0}F_\eps(u;O)\le V_+(u;O)=\int_O\!\lim_{\rho\to 0}\inf\!\left\{\limsup_{\eps\to 0}\frac{\M_\eps(u;\Q)}{\lambda(\Q)}:x\in \Q\in \Bal_o(O)\!,\diam(\Q)\le \!\rho\right\}\!dx.
\end{align*}
The Vitali envelope of a set function in connection with the integral representation of $\Gammap$-limits was introduced in~\cite{bouchitte-fonseca-mascarenhas98} (see also~\cite{bellieud-bouchitte00}). This path has the advantage to avoid any approximations of Sobolev functions by regular ones. It allows, when we assume $p$-growth conditions, to give general results for $\Gammap$-limit and in particular to give a general point of view in homogenization and relaxation problems for Borel measurable integrands $\f(x,v,\xi)$ (see Sect.~\ref{Integral representation with $p$-growth conditions}).

\medskip

Plan of the paper. Sect.~\ref{s1} presents the main assumptions (``global and local" conditions) and the statement of the general results (see Theorem~\ref{main-result-eps} and Theorem~\ref{H-theorem}). Theorem~\ref{H-theorem} is an integral representation result of $\Gammap$-limit, it is a consequence of local conditions (\ref{H0},~\ref{H1} and~\ref{H2}) and Theorem~\ref{main-result-eps}. In Sect.~\ref{Vitali-envelope} we state and prove an integral representation for the Vitali envelope of arbitrary nonnegative set functions. In Sect.~\ref{Proof-main-theorem} we give the proof of Theorem~\ref{main-result-eps} and some other related results. Finally in Sect.~\ref{Integral representation with $p$-growth conditions} we give a general $\Gammap$-convergence result in the $p$-growth case Theorem~\ref{p-growth-theorem}, which can be seen as an extension in a nonconvex (an vectorial) case of Theorem IV in \cite[p. 265]{dalmaso-modica86}. In fact, we show how to verify the local conditions ~\ref{H1} and~\ref{H2} when we deal with $p$-growth, the technics we use are inspired by~\cite{bouchitte-fonseca-mascarenhas98}. In Subsect.~\ref{sub-homog} as an application of Theorem~\ref{p-growth-theorem} we consider a general point of view of the homogenization of functional integral of the calculus of variations. In Subsect.~\ref{subsection-relaxation} we give an extension of the Acerbi-Fusco-Dacorogna relaxation theorem when the integrand is assumed Borel measurable only.

\section{Main results}\label{s1}
\subsection{General framework}
Fix $\alpha\ssup 0$ and $p\in]1,\infty[$. We denote by ${\mathcal I}(p,\alpha)$ the set of functionals $F: W^{1,p}(\Omega;\RR^m)\times\O(\Omega)\to [0,\infty]$ satisfying:\\

\begin{hypc}
\item\label{C0} for every $O\in \O(\Omega)$ and every $u\in\dom F(\cdot;O)$ we have
\begin{align*}
F(u;O)\ge \alpha\Vert \nabla u\Vert_{L^p(O;\RR^m)}^p;
\end{align*}

\medskip

\item\label{C1} for every $u\in \dom F(\cdot;\Omega)$ the set function $F(u;\cdot)$ is the trace on $\O(\Omega)$ of a Borel measure absolutely continuous with respect to the Lebesgue measure $\lambda$ on $\Omega$;\\

\item\label{C2} for every $O\in\O(\Omega)$ the functional $F(\cdot;O)$ is local, i.e., if $u=v$ a.e. in $O$ then $F(u;O)=F(v;O)$ for all $u,v\in \dom F(\cdot;O)$.
\end{hypc}

\bigskip

Consider a family $\FF:=\{F_\eps\}_{\eps\in ]0,1]}$ of functionals $F_\eps: W^{1,p}(\Omega;\RR^m)\times\O(\Omega)\to [0,\infty]$.
For each $O\in \O(\Omega)$ and each $u\in L^p(\Omega;\RR^m)$ we set
\begin{align*}
\FF_-(u;O)&:=\inf\left\{\liminf_{\eps\to 0} F_\eps(u_\eps;O):u_\eps\to u \mbox{ in }L^p(\Omega;\RR^m)\right\};\\
\FF_+(u;O)&:=\inf\left\{\limsup_{\eps\to 0} F_\eps(u_\eps;O):u_\eps\to u \mbox{ in }L^p(\Omega;\RR^m)\right\}.
\end{align*}
The functional $\FF_-(\cdot;O)$ (resp. $\FF_+(\cdot;O)$) is the $\Gammap$-$\liminf_{\eps\to0}$ (resp. $\Gammap$-$\limsup_{\eps\to0}$) of the family $\FF(\cdot;O)=\{F_\eps(\cdot;O)\}_\eps$ . If $u\in W^{1,p}(\Omega;\RR^m)$ and $\FF_+(u;O)=\FF_-(u;O)$ then we say that $\FF(\cdot;O)$ $\Gammap$-converges at $u$ to the $\Gammap$-limit $\FF_0(u;O):=\FF_+(u;O)=\FF_-(u;O)$.

\bigskip

%\subsection{Local Dirichlet minimization problems associated to $\FF\subset\I(p,\alpha)$}
We associate to $\FF=\{F_\eps\}_{\eps\in ]0,1]}\subset\I(p,\alpha)$ a family of local Dirichlet problems $\{\M_\eps\}_{\eps\in ]0,1]},\;\;\M_\eps: W^{1,p}(\Omega;\RR^m)\times\O(\Omega)\to [0,\infty]$ defined by
\begin{align*}
\M_{\eps}(u;O):=\inf\big\{F_\eps(v;O):W^{1,p}(\Omega;\RR^m)\ni v=u\text{ in }\Omega\setminus O\big\}.
\end{align*}
Note that we can write
\begin{align*}
\M_{\eps}(u;O)=\inf\big\{F_\eps(v;O): v\in u+W^{1,p}_0(O;\RR^m)\big\}
\end{align*}
since $u+W^{1,p}_0(O;\RR^m)=\left\{v\in W^{1,p}(\Omega;\RR^m): v-u=0 \text{ in }\Omega\setminus O\right\}$ (see \cite[p. 234, Theorem 9.1.3]{adams-hedberg}).
\begin{remark} The functional $\M_\eps(\cdot;O)$ can be seen as the ``quotient functional" $\widetilde{F_\eps}(\cdot;O)$ defined on the quotient space of $W^{1,p}(\Omega;\RR^m)$ by $W^{1,p}_0(O;\RR^m)$, i.e.,
\begin{align*}
\widetilde{F_\eps}(\cdot;O)&:\displaystyle\sfrac{\displaystyle W^{1,p}(\Omega;\RR^m)}{\displaystyle W^{1,p}_0(O;\RR^m)}\to [0,\infty]\\ \\
\mbox{ with }\widetilde{F_\eps}([u];O)&:=\inf_{v\in [u]}F_\eps(v;O)=\M_\eps(u;O)
\end{align*}
where $[u]=u+W^{1,p}_0(O;\RR^m)$ is the equivalent class of $u$.
\end{remark}
\subsection{A general $\Gammap$-convergence theorem}
We denote by $\lambda$ the Lebesgue measure on $\Omega$. For each $O\in\O(\Omega)$ we denote by $\mesabs(O)$ the space of nonnegative {\em finite} Borel measures on $O$ which are absolutely continuous with respect to the Lebesgue measure $\lambda\lfloor_O$ on $O$. 

\medskip

Let us introduce the {\em {\rm M}-sets associated to $\FF$}: for each $O\in \O(\Omega)$ we set
\begin{align}\label{Msets}
{\rm M}_\FF(O):=
\bigg\{
u\in W^{1,p}(\Omega;\RR^m):\exists \mu_u\in \mesabs(O)\quad
\sup_{\eps>0}\M_\eps(u;\cdot)\le\mu_u(\cdot)\text{ on }O
\bigg\}.
\end{align}

\begin{itemize}
\item We assume that all the affine maps, i.e., functions of the form $u(x)=v+\zeta x$ with $(x,v,\zeta)\in\Omega\times\RR^m\times\MM^{m\times d}$, belong to ${\rm M}_\FF(O)$.
\end{itemize}
We will see in Theorem~\ref{H-theorem} that the {\rm M}-set is the set where an integral representation of the $\Gammap$-limit is possible.

\bigskip

To the family $\FF$ we associate $\M_+:W^{1,p}(\Omega;\RR^m)\times\O(\Omega)\to [0,\infty]$ defined by
\begin{align*}
\M_+(u;O):=\limsup_{\eps\to 0}\M_\eps(u;O).
\end{align*}

The following result provides bounds in integral forms of both $\Gammap$-$\liminf_{\eps\to0}$ and $\Gammap$-$\limsup_{\eps\to0}$ of a family $\FF=\{F_\eps\}_{\eps\in ]0,1]}\subset {\mathcal I}(p,\alpha)$, i.e., satisfying~\ref{C0},~\ref{C1} and~\ref{C2}.
\begin{theorem}\label{main-result-eps} Let $\FF=\{F_\eps\}_{\eps\in ]0,1]}\subset {\mathcal I}(p,\alpha)$ and let $(u,O)\in W^{1,p}(\Omega;\RR^m)\times\O(\Omega)$. 
\begin{hyp2}
\item\label{upper bound} If $u\in{\rm M}_\FF(O)$ then
\begin{align*}
\FF_+(u;O)\le\FF_+^\D(u;O)\le\int_O\lim_{\rho\to 0}\inf\left\{\frac{\M_{+}(u;\Q)}{\lambda(\Q)}: x\in \Q\in \Bal_o(O),\;\diam(\Q)\le \rho\right\}dx
\end{align*}
where 
\[
\FF_+^\D(u;O):=\inf\left\{\limsup_{\eps\to 0}F_\eps(v_\eps;O):W^{1,p}_0(O;\RR^m)+u\ni v_\eps\to u\mbox{ in }L^p(\Omega;\RR^m)\right\};
\]
\item\label{Lower Bound} There exists $\{u_{\eps_n}\}_{n}\subset W^{1,p}(\Omega;\RR^m)$ with $\sup_n F_{\eps_n}(u_{\eps_n};O)\sinf\infty$ such that $u_{\eps_n}\to u$ in $L^p(\Omega;\RR^m)$ as $n\to \infty$ and
\begin{align*}
\FF_-^\D(u;O)\ge\FF_-(u;O)\ge \int_O\lim_{\rho\to 0} \lim_{n\to \infty}\frac{F_{\eps_n}(u_{\eps_n};\Q_\rho(x))}{\rho^d}dx%\quad\text{ a.e. in }O.
\end{align*}
where 
\[
\FF_-^\D(u;O):=\inf\left\{\liminf_{\eps\to 0}F_\eps(v_\eps;O):W^{1,p}_0(O;\RR^m)+u\ni v_\eps\to u\mbox{ in }L^p(\Omega;\RR^m)\right\}.
\]
\end{hyp2}
\end{theorem}
\begin{remark}\label{domain-M} If $u\in {\rm M}_\FF(O)$ then there exists $\mu_u\in \mesabs(O)$ such that $
\sup_{\eps>0}\M_\eps(u;\cdot)\le\mu_u(\cdot)\text{ on }O$. Therefore we have
\begin{align*}
\M_+(u;\cdot)\le\mu_u(\cdot)\text{ on }O.
\end{align*}
Taking account of Theorem~\ref{main-result-eps}~\ref{upper bound} we deduce that $\FF_+(u;O)\sinf\infty$, which means that
\begin{align*}
{\rm M}_\FF(O)\subset\dom\FF_+^\D(\cdot;O):=\left\{u\in W^{1,p}(\Omega;\RR^m):\FF_+^\D(u;O)\sinf\infty\right\}\subset \dom\FF_+(\cdot;O).
\end{align*}
\end{remark}
%\subsection{Local conditions for the integral representation} 

\bigskip

To the family $\FF=\{F_\eps\}_{\eps\in ]0,1]}$ we associate $\M_-:W^{1,p}(\Omega;\RR^m)\times\O(\Omega)\to [0,\infty]$ defined by
\begin{align*}
\M_-(u;O):=\liminf_{\eps\to 0}\M_\eps(u;O).
\end{align*}

\medskip

Let $O\in\O(\Omega)$ and let $u\in W^{1,p}(\Omega;\RR^m)$. We denote the affine tangent map of $u$ at $x\in O$ by
\[
u_x(\cdot):=u(x)+\nabla u(x)(\cdot-x).
\]

Consider the following local inequalities for $u\in {\rm M}_\FF(O)$:

\bigskip

\begin{hyph}
\item\label{H0} $\displaystyle\liminf_{\rho\to 0}\frac{\M_-(u_x;\Q_\rho(x))}{\rho^d} \ge \limsup_{\rho\to 0}\frac{\M_+(u_x;\Q_\rho(x))}{\rho^d}\quad\text{ a.e. in }O$;\\ \\

\item\label{H1} $\displaystyle\limsup_{\rho\to 0}\liminf_{\eps\to 0}\frac{F_\eps(u_\eps;\Q_\rho(x))}{\rho^d} \ge \liminf_{\rho\to 0}\frac{\M_-(u_x;\Q_\rho(x))}{\rho^d}\quad\text{ a.e. in }O\;$ for all $\{u_\eps\}_\eps\subset W^{1,p}(\Omega;\RR^m)$ such that $u_\eps\to u$ in $L^p(\Omega;\RR^m)$ and $\sup_\eps F_\eps(u_\eps;O)<\infty$;\\ \\

\item\label{H2} $\displaystyle\limsup_{\rho\to 0}\frac{\M_+(u_x;\Q_\rho(x))}{\rho^d} \ge \liminf_{\rho\to 0}\frac{\M_+(u;\Q_\rho(x))}{\rho^d}\quad\text{ a.e. in }O.$
\end{hyph}

\medskip

\begin{remark}\label{remarks-conditions-local} We make some remarks on the previous inequalities.
\begin{hyp2}
\item Condition~\ref{H0}, related to the integral representation of the $\Gammap$-limit of functionals of the calculus of variations, is already known when $p$-polynomial growth (and convexity conditions) is assumed see \cite[p. 451]{dalmaso2006}. 
 \item The condition~\ref{H1} (resp.~\ref{H2}) can be seen as a ``local" $\Gammap$-$\liminf$ (resp. $\Gammap$-$\limsup$)  inequality. To verify inequality~\ref{H2} (resp.~\ref{H1}) we need to replace $u$ (resp. a sequence $\{u_\eps\}_\eps$ converging in $L^p$ to $u$ and satisfying $\sup_\eps F_\eps(u_\eps;O)<\infty$) by the affine tangent map $u_x$ in the localization of $\M_\eps$ on ``small" cubes $\Q_\rho(x)$. This can be performed, for instance, by using growth conditions see Sect.~\ref{Integral representation with $p$-growth conditions}.
\end{hyp2}
\end{remark}

\medskip

The following lemma is used in the proof of Theorem~\ref{H-theorem} and its proof is given in Sect.~\ref{Proof-main-theorem}.
\begin{lemma}\label{lemma-H-theorem} Let $\FF=\{F_\eps\}_{\eps\in ]0,1]}\subset\I(p,\alpha)$. Let $O\in\O(\Omega)$ and let $u\in {\rm M}_\FF(O)$. If~\ref{H0},~\ref{H1} and~\ref{H2} hold then 
%the following limit $\lim_{\rho\to 0}\frac{\M_+(u_x;\Q_\rho(x))}{\rho^d}$ exists almost everywhere in $O$. Moreover 
the function $O\ni x\mapsto\limsup_{\rho\to 0}\frac{\M_+(u_x;\Q_\rho(x))}{\rho^d}$ is measurable and satisfies
\begin{align*}
\limsup_{\rho\to 0}\frac{\M_+(u_x;\Q_\rho(x))}{\rho^d}=\liminf_{\rho\to 0}\frac{\M_-(u_x;\Q_\rho(x))}{\rho^d}\quad\mbox{ a.e. in }O.
\end{align*}
\end{lemma}
Here is the general $\Gammap$-convergence theorem which shows that under the local inequalities~\ref{H0},~\ref{H1} and~\ref{H2} the family $\FF(\cdot;O)$ $\Gammap$-converges to an integral functional of the calculus of variations at every $u\in {\rm M}_\FF(O)$. In Sect.~\ref{Integral representation with $p$-growth conditions} we give applications to homogenization and relaxation of this result. When $\FF_+^\D=\FF_-^\D$ we denote by $\FF_0^\D=\FF_+^\D=\FF_-^\D$ the common value.
\begin{theorem}\label{H-theorem} Let $\FF=\{F_\eps\}_{\eps\in ]0,1]}\subset\I(p,\alpha)$. Let $O\in\O(\Omega)$ and let $u\in {\rm M}_\FF(O)$. If~\ref{H0},~\ref{H1} and~\ref{H2} hold then the family of functionals $\FF(\cdot;O)$ $\Gammap$-converges at $u$ to 
\begin{align}\label{formula-int-representation}
\FF_0(u;O)=\FF_0^\D(u;O)=\int_O\limsup_{\rho\to 0}\frac{\M_+(u_x;\Q_\rho(x))}{\rho^d}dx.
\end{align}
Moreover, we have for almost all $x\in O$
\begin{align}\label{formula-int-representation-integrands}
\limsup_{\rho\to 0}\frac{\M_+(u_x;\Q_\rho(x))}{\rho^d}=\liminf_{\rho\to 0}\frac{\M_-(u_x;\Q_\rho(x))}{\rho^d}.
\end{align}
\end{theorem}
\begin{proof} Let $O\in\O(\Omega)$ and let $u\in {\rm M}_\FF(O)$. From Theorem~\ref{main-result-eps}~\ref{Lower Bound}, there exists $\{u_{\eps}\}_{\eps}\subset W^{1,p}(\Omega;\RR^m)$ with $\sup_\eps F_{\eps}(u_{\eps};O)\sinf\infty$ such that $u_{\eps}\to u$ in $L^p(\Omega;\RR^m)$ as $\eps\to 0$ and
\begin{align*}
\FF_-^\D(u;O)\ge\FF_-(u;O)\ge \int_O\lim_{\rho\to 0} \lim_{\eps\to 0}\frac{F_{\eps}(u_{\eps};\Q_\rho(x))}{\rho^d}dx%\quad\text{ a.e. in }O.
\end{align*}
Using the local inequalities~\ref{H0},~\ref{H1} and~\ref{H2}, Lemma~\ref{lemma-H-theorem} together with Theorem~\ref{main-result-eps} we have
\begin{align*}
\FF_+(u;O)\le\FF_+^\D(u;O)&\le \int_O\lim_{\rho\to 0}\inf\left\{\frac{\M_{+}(u;\Q)}{\lambda(\Q)}: x\in \Q\in \Bal_o(O),\;\diam(\Q)\le \rho\right\}dx\\
&\le\int_O\limsup_{\rho\to 0}\frac{\M_{+}(u;\Q_\rho(x))}{\rho^d}dx\\
%&\le \int_O\lim_{\rho\to 0}\frac{\M_{+}(u;\Q_\rho(x))}{\rho^d}dx\\
%&\le\int_O\lim_{\rho\to 0}\frac{\M_+(u_x;\Q_\rho(x))}{\rho^d}dx\\
&=\int_O\liminf_{\rho\to 0}\frac{\M_-(u_x;\Q_\rho(x))}{\rho^d}dx\\
&\le \int_O\lim_{\rho\to 0}\lim_{\eps\to 0}\frac{F_{\eps}(u_{\eps};\Q_\rho(x))}{\rho^d}dx\\
&\le \FF_-(u;O)\le\FF_-^\D(u;O).
\end{align*}
Thus~\eqref{formula-int-representation} holds. The equality~\eqref{formula-int-representation-integrands} is a consequence of Lemma~\ref{lemma-H-theorem}. \mbox{\qedhere}
\end{proof}

\subsection{The relaxation case}
We examine the particular case of a constant family with respect to the parameter $\FF=\{F_\eps=F\}_\eps\subset \I(p,\alpha)$. We set for every $(u,O)\in W^{1,p}(\Omega;\RR^m)\times\O(\Omega)$
\begin{align*}
\FF_0(u;O)&:=\inf\left\{\liminf_{\eps\to 0}F(v_\eps;O):W^{1,p}(\Omega;\RR^m)\ni v_\eps\to u\mbox{ in }L^p(\Omega;\RR^m)\right\};\\
\FF^\D_0(u;O)&:=\inf\left\{\liminf_{\eps\to 0}F(v_\eps;O):W^{1,p}_0(O;\RR^m)+u\ni v_\eps\to u\mbox{ in }L^p(\Omega;\RR^m)\right\};\\
\M(u;O)&:=\inf\left\{F(v;O):v\in u+W^{1,p}_0(O;\RR^m)\right\}.
\end{align*}
The following abstract relaxation result is a direct consequence of Theorem~\ref{H-theorem}.
\begin{proposition}\label{abstract-constant-dirichlet} Let $O\in\O(\Omega)$ and let $u\in {\rm M}_\FF(O)$. If~\ref{H0},~\ref{H1} and~\ref{H2} hold then  
\begin{align}\label{equality-fd}
\FF_0(u;O)=\FF^\D_0(u;O)=\int_O\lim_{\rho\to 0}\frac{\M(u_x;\Q_\rho(x))}{\rho^d}dx.
\end{align}
\end{proposition}
\begin{remark} In particular~\eqref{equality-fd} holds for all $u\in \dom F(\cdot;O)$ since it is easy to see that $\dom F(\cdot;O)\subset {\rm M}_\FF(O)$.
\end{remark}
\subsection{Remarks on the limit integrand} We assume that the assumptions of Theorem~\ref{H-theorem} hold. We give descriptions of the limit integrand $\f_0$ by considering some particular cases.
\begin{hyp2}
\item If we define $\ftild_0:\Omega\times\RR^d\times\MM^{m\times d}\to [0,\infty]$ by
\begin{align}\label{generic-formula-tilde}
\ftild_0(x,v,\xi):=\limsup_{\rho\to 0}\frac{\M_+(v+\xi(\cdot-x);\Q_\rho(x))}{\rho^d},
\end{align}
and for each $u\in {\rm M}_\FF(O)$
\begin{align}\label{generic-formula}
\f_0(x,u(x),\nabla u(x)):=\limsup_{\rho\to 0}\frac{\M_+(u_x;\Q_\rho(x))}{\rho^d}
\end{align}
then the formula~\eqref{formula-int-representation} becomes
\begin{align*}
\FF_0(u;O)=\int_O \ftild_0(x,u(x),\nabla u(x))dx.
\end{align*}
Indeed, we have for every $x\in O$
\begin{align}\label{l=ltild}
\f_0(x,u(x),\nabla u(x))=\ftild_0(x,u(x),\nabla u(x)).
\end{align}
In fact, we do not know whether the integrand $\ftild_0$ is Borel measurable. Because of the equality~\eqref{l=ltild}, the function $O\ni x\mapsto \ftild_0(x,u(x),\nabla u(x))$ is measurable for all $u\in{\rm M}_\FF(O)$.

\bigskip

\item Assume that $\{F_\eps\}_\eps=\FF$ is given under integral form, i.e.,  $F_\eps:W^{1,p}(\Omega;\RR^m)\times\O(\Omega)\to [0,\infty]$ is defined by
\begin{align*}
F_\eps(u;O):=\int_O L_\eps(x,u(x),\nabla u(x))dx
\end{align*}
where $\f_\eps:\Omega\times\RR^d\times\MM^{m\times d}\to [0,\infty]$ is Borel measurable for all $\eps\in ]0,1]$. Then
\begin{align}\label{formula-L0-eps}
&\ftild_0(x,v,\xi)=\\
&\limsup_{\rho\to 0}\limsup_{\eps\to 0}\inf\left\{\!\!\fint_{\Q_\rho(x)}\!\!\!\!\f_\eps(y,v\!+\!\xi(y-x)\!+\!\varphi(y),\xi\!+\!\nabla\varphi(y)\!)\!dy:\varphi\in W^{1,p}_0(\Q_\rho(x)\!;\!\RR^m)\!\!\right\}\!\!.\notag
\end{align}
If, moreover, we assume that $\f_\eps$ does not depend of the variable $v$, i.e., $\f_\eps:\Omega\times\MM^{m\times d}\to [0,\infty]$ then~\eqref{l=ltild} becomes 
\begin{align*}
\f_0(x,\nabla u(x))=\ftild_0(x,\nabla u(x))
\end{align*}
for all $x\in O$ and all $u\in {\rm M}_\FF(O)$. Since the affine functions belong to ${\rm M}_\FF(O)$ we deduce that for every $x\in O$ and every $\xi\in\MM^{m\times d}$
\begin{align*}
\f_0(x,\xi)=\ftild_0(x,\xi).
\end{align*}

\item Now, we consider the case where $\{F_\eps=F\}_\eps=\FF$ is constant with respect to $\eps$ and $F:W^{1,p}(\Omega;\RR^m)\times\O(\Omega)\to [0,\infty]$ is defined by
\begin{align*}
F(u;O):=\int_O L(x,u(x),\nabla u(x))dx
\end{align*}
where $L:\Omega\times\RR^d\times\MM^{m\times d}\to [0,\infty]$ is Borel measurable. If we define for every $(x,v,\xi)\in\Omega\times\RR^d\times\MM^{m\times d}$
\begin{align}\label{formula-L0}
&\ftild_0(x,v,\xi)=\\
&\limsup_{\rho\to 0}\inf\left\{\fint_{\Q_\rho(x)}\f(y,v\!+\!\xi(y-x)\!+\!\varphi(y),\xi\!+\!\nabla\varphi(y)\!)\!dy:\varphi\in W^{1,p}_0(\Q_\rho(x);\RR^m)\right\}.\notag
\end{align}
then for every $u\in {\rm M}_\FF(O)$ and every $x\in O$
\begin{align*}
\f_0(x,u(x),\nabla u(x))=\ftild_0(x,u(x),\nabla u(x)).
\end{align*}
We will show in Proposition~\ref{Reduction formula for Q_pf} that when $L$ is Carath\'eodory with $p$-growth and $p$-coercivity we recover the classical quasiconvex envelope~\cite[Theorem 9.8, p. 432]{dacorogna08}.
\end{hyp2}
%%%%%%%Section-2%%%%%
\section{Integral representation of Vitali envelope and derivation of set functions}\label{Vitali-envelope}
\subsection{Integral representation of Vitali envelope of set functions}
For a given open set $O\subset \Omega$ we denote by $\Bal_o(O)$ the set of all open cube of $O$. We denote by ${\mathcal Q}_c(O)$ the set of all closed cube of $O$. 

Let $G:\Bal_o(\Omega)\to ]-\infty,\infty]$ be a set function. We define the {\em Vitali envelope} of $G$ with respect to $\lambda$
\begin{align*}
\O(\Omega)\ni O\mapsto V_G(O):=\sup_{\eps>0}\inf\left\{ \sum_{i\in I}G(\Q_i):\left\{\overline{\Q}_i\right\}_{i\in I}\in \V^\eps(O)\right\}
\end{align*}
where for any $\eps\ssup 0$
\begin{align*}
\V^\eps(O):=
\left\{
\left\{\overline{\Q}_i\right\}_{i\in I}\subset\Bal_c(\Omega):I\text{ is countable, }\lambda(O\setminus \mathop{\cup}_{i\in I}{\Q}_i)=0,\right.\;\overline{\Q}_i\subset O&
\\
\diam(\Q_i)\in ]0,\eps[\;\text{ and }\;\overline{\Q}_i\cap \overline{\Q}_j=\emptyset\;\text{ for all }\;i\not=&j\bigg\}.
\end{align*} 
\begin{remark}\label{trivial-vitali}
If $G$ is the trace on $\Bal_o(\Omega)$ of a positive Borel measure $\nu$ on $\Omega$ which is absolutely continuous with respect to $\lambda$ then $V_G(O)=\nu(O)$ for all $O\in\O(\Omega)$.
\end{remark}
Let $G:\Bal_o(\Omega)\to ]-\infty,\infty]$ be a set function. 
Define the upper and the lower derivatives at $x\in\Omega$ of $G$ with respect to $\lambda$ as follows
\begin{align*}
\Di_\lambda G(x)&:=\lim_{\rho\to 0}\inf\left\{\frac{G(\Q)}{\lambda(\Q)}:x\in \Q\in \Bal_o(\Omega),\;\diam(\Q)\le \rho\right\};\\
\Ds_\lambda G(x)&:=\lim_{\rho\to 0}\sup\left\{\frac{G(\Q)}{\lambda(\Q)}:x\in \Q\in \Bal_o(\Omega),\;\diam(\Q)\le \rho\right\}.\\
\end{align*}
We say that $G$ is $\lambda$-differentiable in $O$ if for $\lambda$-almost all $x\in O$ it holds 
\begin{align*}
-\infty\sinf\Di_\lambda G(x)=\Ds_\lambda G(x)\sinf\infty.
\end{align*}  
%In this case we denote the common value by $D_\lambda G(x)$ and 
%\begin{align*}
%D_\lambda G(x)=\lim_{\rho\to 0}\frac{G(\Q_\rho(x))}{\lambda(\Q_\rho(x))}.
%\end{align*}
\begin{remark} For every $O\in\O(\Omega)$ and every $x\in O$ we have
\begin{align*}
\Di_\lambda G(x)&:=\lim_{\rho\to 0}\inf\left\{\frac{G(\Q)}{\lambda(\Q)}:x\in \Q\in \Bal_o(O),\;\diam(\Q)\le \rho\right\};\\
\Ds_\lambda G(x)&:=\lim_{\rho\to 0}\sup\left\{\frac{G(\Q)}{\lambda(\Q)}:x\in \Q\in \Bal_o(O),\;\diam(\Q)\le \rho\right\}.\\
\end{align*}
\end{remark}
The proof of the following classical result can be found in Appendix.
\begin{lemma}\label{measurability} The functions $\Di_\lambda G(\cdot)$ and $\Ds_\lambda G(\cdot)$ are $\lambda$-measurable.
\end{lemma}
\begin{remark}\label{diff-borel-measure} When $G=\nu$ is a Borel finite measure absolutely continuous with respect to $\lambda$ then $\nu$ is $\lambda$-differentiable in $O$ and 
\begin{align*}
D_\lambda \nu(x)=\lim_{\rho\to 0}\frac{\nu(\Q_\rho(x))}{\rho^d}\quad\mbox{ a.e. in }O.
\end{align*}
\end{remark}
The following result establishes an integral representation for the Vitali envelope of nonnegative set functions.
\begin{proposition}\label{inequality-Vitali-envelope}
Let $H:\Bal_o(\Omega)\to [0,\infty]$ be a set function. For every $O\in\O(\Omega)$ we have
\begin{align*}
V_H(O)=\int_O \Di_\lambda H(y)dy.
\end{align*}
\end{proposition}
The following lemma was inspired by reading \cite{bongiorno81}.
\begin{lemma}\label{abstract-diff} Let $G:\Bal_o(\Omega)\to ]-\infty,\infty]$ be a set function and $O\in\O(\Omega)$. 
\begin{hyp1}
\item\label{Fund1} If $\Di_\lambda G(x)\le 0$ $\lambda$-a.e. in $O$ then $V_G(O)\le 0$. 
\item\label{Fund2} If $\Di_\lambda G(x)\ge 0$ $\lambda$-a.e. in $O$ then $V_G(O)\ge 0$. 
\item\label{Fund3} If $\Di_\lambda G(x)= 0$ $\lambda$-a.e. in $O$ then $V_G(O)= 0$. 
\end{hyp1}
\end{lemma}
\begin{proof} The assertion~\ref{Fund3} is a consequence of~\ref{Fund1} and~\ref{Fund2}. \\

{\em Proof of~\ref{Fund1}.} It is enough to show that for every $\eps\ssup 0$ if 
\begin{align}\label{abstract-diff-eq1}
\Di_\lambda G(x)\sinf\eps\quad\lambda\text{-a.e. in }O
\end{align}
then
\[
\inf\left\{\sum_{i\in I}G(\Q_i):\left\{\overline{\Q}_i\right\}_{i\in I}\in\V_{\eps}(O)\right\}\sinf\eps\lambda(O).
\]
Fix $\eps\ssup 0$. Let $N\subset O$ with $\lambda(N)=0$ be such that $O\setminus N=\left[ \Di_\lambda G(\cdot)\sinf\eps\right]$. Using Lemma~\ref{consequenceSh-0}~\ref{consequenceSh-1} with $h=\eta=\eps$ and $S_h=O\setminus N$, there exists a {countable} pairwise disjointed family $\{\Q_i\}_{i\in I}\subset\Bal_o(O)$ such that
\begin{align}\label{abstract-diff-eq43}
\lambda((O\setminus N)\setminus \mathop{\cup}_{i\in I}\Q_i)=0,\quad \forall i\in I\quad G(\Q_i)\sinf \eps\lambda(\Q_i)\mbox{ and }\;\diam(\Q_i)\sinf\eps.
\end{align}
From \eqref{abstract-diff-eq43} we have $\lambda(O\ssetminus \mathop{\cup}_{i\in I}\overline{\Q}_i)=0$ since $\lambda(N)=0$. Consequently, we have $\sum_{i\in I}\lambda(\overline{\Q}_i)=\sum_{i\in I}\lambda({\Q}_i)=\lambda(O)$. Summing over $i\in I$ the first inequality~\eqref{abstract-diff-eq43} we obtain
\begin{align*}
\inf\left\{\sum_{i\in I}G(\Q_i):\left\{\overline{\Q}_i\right\}_{i\in I}\in\V_{\eps}(O)\right\}\sinf \eps\sum_{i\in I}\lambda({\Q}_i) =\eps\lambda(O).\\
\end{align*}

{\em Proof of~\ref{Fund2}.} Let $\{\eps_n\}_{n\in\NN}\subset ]0,1[$ be such that $\lim_{n\to\infty}\eps_n=0$. Let $n\in\NN$. There exists $\{\Q_i^n\}_{i\in I_n}\in \V_{\eps_n}(O)$ such that
\begin{align}\label{V-positif}
\inf\left\{\sum_{i\in I}G(\Q_i):\{\Q_i\}\in\V_{\eps_n}(O)\right\}+\frac1n\ge\sum_{i\in I_n}\frac{G(\Q_i^n)}{\lambda(\Q_i^n)}\lambda(\Q_i^n).
\end{align}
Fix $x\in O$ be such that
\begin{align*}
\Di_\lambda G(x)\ge 0\quad\mbox{ and }\quad x\notin\mathop{\bigcup}_{n\in\NN}(O\setminus \mathop{\cup}_{i\in I_n}\Q_i^n).
\end{align*}
There exists $i_x^n\in I_n$ such that $x\in \Q_{i_x^n}^n$. From~\eqref{V-positif} it follows that
\begin{align}\label{V-positif-1}
&\inf\left\{\sum_{i\in I}G(\Q_i):\{\Q_i\}\in\V_{\eps_n}(O)\right\}+\frac1n\\
&\ge\frac{G(\Q_{i_x^n}^n)}{\lambda(\Q_{i_x^n}^n)}\lambda(\Q_{i_x^n}^n)\notag\\
&\ge \inf\left\{\frac{G(\Q)}{\lambda(\Q)}:x\in \Q\in \Bal_o(O),\;\diam(\Q)\le \eps_n\right\}\lambda(\Q_{i_x^n}^n).\notag
\end{align}
Passing to the limit $n\to\infty$ in~\eqref{V-positif-1} we obtain
\begin{align*}
V_G(O)\ge \Di_\lambda G(x)\liminf_{n\to\infty}\lambda(\Q_{i_x^n}^n)\ge0.
\end{align*}
The proof is complete.\mbox{\qedhere}
\end{proof}

\begin{corollary}\label{shouldbtheorem} If $G=H-\nu$ where $H:\Bal_o(\Omega)\to ]-\infty,\infty]$ is a set function, $\nu$ is a finite Borel measure on $O\in\O(\Omega)$ absolutely continuous with respect to $\lambda\lfloor_O$, and if 
\begin{align*}
\Di_\lambda G(x)=0\quad\lambda\mbox{-a.e. in } O
\end{align*} 
then
\begin{align*}
V_H(O)=\nu(O).
\end{align*}
\end{corollary}
\begin{proof} Use Lemma~\ref{abstract-diff}\ref{Fund3} and remark that $V_G(O)=V_H(O)-\nu(O)$.\mbox{\qedhere}
\end{proof}
\subsubsection*{\bf Proof of Proposition~\ref{inequality-Vitali-envelope}} Assume that $\Di_\lambda H\in L^1(O)$, i.e.,
\begin{align*}
\int_O \Di_\lambda  H(y)dy<\infty.
\end{align*}
Define $G:\Bal_o(O)\to ]-\infty,\infty]$ by
\begin{align*}
G(\Q):=H(\Q)-\int_\Q \Di_\lambda  H(y)dy.
\end{align*}
If we show that $\Di_\lambda G=0$ a.e. in $O$ then by Corollary~\ref{shouldbtheorem} we can conclude that
\begin{align*}
V_H(O)=\nu(O)\mbox{ with }\nu:= \Di_\lambda  H(\cdot)\lambda\lfloor_{O}.
\end{align*}
Fix $x\in O$ such that
\begin{align}
\label{lebesguepoint1}\Di_\lambda  H(x)&=\lim_{\rho\to 0}\fint_{\Q_\rho(x)} \Di_\lambda  H(y)dy<\infty;\\
\label{lebesguepoint2}\Di_\lambda\nu(x)&=\Di_\lambda  H(x)\sinf\infty.
\end{align}
We set $\mathfrak{B}_{x,\rho}:=\left\{\Q:x\in\Q\in\Bal_o(O)\mbox{ and }\diam(\Q)\le\rho\right\}$ for all $\rho\ssup0$. On one hand, we have for every $\rho\ssup0$ and every $\Q\in \mathfrak{B}_{x,\rho}$
\begin{align*}
\frac{G(\Q)}{\lambda(\Q)}&\le \frac{H(\Q)}{\lambda(\Q)}+\sup\left\{-\fint_\Q \Di_\lambda  H(y)dy:\Q\in\mathfrak{B}_{x,\rho}\right\}\\
&=\frac{H(\Q)}{\lambda(\Q)}-\inf\left\{\fint_\Q \Di_\lambda  H(y)dy:\Q\in\mathfrak{B}_{x,\rho}\right\}.
\end{align*}
Taking the infimum over every $\Q\in\mathfrak{B}_{x,\rho}$ we obtain
\begin{align*}
&\inf\left\{\frac{G(\Q)}{\lambda(\Q)}:\Q\in\mathfrak{B}_{x,\rho}\right\}\\
&\le\inf\left\{\frac{H(\Q)}{\lambda(\Q)}:\Q\in\mathfrak{B}_{x,\rho}\right\}-\inf\left\{\fint_\Q \Di_\lambda  H(y)dy:\Q\in\mathfrak{B}_{x,\rho}\right\}.
\end{align*}
Letting $\rho\to 0$ and using~\eqref{lebesguepoint1} and~\eqref{lebesguepoint2}, we have
\begin{align}\label{DGs}
\Di_\lambda G(x)\le \Di_\lambda H(x)-\Di_\lambda H(x)=0.
\end{align}
On the other hand, we have for every $\rho\ssup0$ and every $\Q\in\mathfrak{B}_{x,\rho}$
\begin{align*}
\frac{G(\Q)}{\lambda(\Q)}&\ge \frac{H(\Q)}{\lambda(\Q)}+\inf\left\{-\fint_\Q \Di_\lambda  H(y)dy:\Q\in\mathfrak{B}_{x,\rho}\right\}\\
&=\frac{H(\Q)}{\lambda(\Q)}-\sup\left\{\fint_\Q \Di_\lambda  H(y)dy:\Q\in\mathfrak{B}_{x,\rho}\right\}.
\end{align*}
Taking the infimum over every $\Q\in\mathfrak{B}_{x,\rho}$ we obtain
\begin{align*}
&\inf\left\{\frac{G(\Q)}{\lambda(\Q)}:\Q\in\mathfrak{B}_{x,\rho}\right\}\\
&\le\inf\left\{\frac{H(\Q)}{\lambda(\Q)}:\Q\in\mathfrak{B}_{x,\rho}\right\}-\sup\left\{\fint_\Q \Di_\lambda  H(y)dy:\Q\in\mathfrak{B}_{x,\rho}\right\}.
\end{align*}
Letting $\rho\to 0$ and using~\eqref{lebesguepoint1} and~\eqref{lebesguepoint2}, we have
\begin{align}\label{DGi}
\Di_\lambda G(x)\ge \Di_\lambda H(x)-\Di_\lambda H(x)=0.
\end{align}
Taking account of~\eqref{DGs} and~\eqref{DGi}, we finally obtain that $\Di_\lambda G(x)=0$. \\

Now, we do not assume that $\Di_\lambda H\in L^1(O)$, in this case the following inequality is always true
\begin{align*}
V_H(O)\le \int_O\Di_\lambda H(y)dy.
\end{align*}
It remains to prove the opposite inequality. For every $n\in\NN$ we set $H_n:\Bal_0(\Omega)\to [0,\infty[$ defined by
\begin{align*}
H_n(\Q):=\left\{\begin{array}{cl}
H(\Q)&\mbox{ if }H(\Q)\le n\lambda(\Q)\\ \\
n\lambda(\Q)&\mbox{ if }H(\Q)\ssup n\lambda(\Q).
\end{array}
\right.
\end{align*}
It is easy to see that 
\begin{align}
\label{mono-Hn}&\forall \Q\in\Bal_0(\Omega)\qquad H_0(\Q)\le H_1(\Q)\le\dots\le H_n(\Q)\le\dots\le\sup_{n\in\NN}H_n(\Q)\le H(\Q);\\
&\forall n\in\NN\qquad\quad\quad\quad\Di_\lambda H_n\le n.\notag
\end{align}
So $\{\Di_\lambda H_n\}_{n\in\NN}\subset L^1(O)$, we apply the first part of the proof to have
\begin{align*}
\forall n\in\NN\qquad V_{H_n}(O)=\int_O \Di_\lambda H_n(y)dy\le V_H(O).
\end{align*}
Using~\eqref{mono-Hn} and monotone convergence theorem we have 
\begin{align}\label{monotone-VH}
\sup_{n\in\NN}V_{H_n}(O)=\int_O \sup_{n\in\NN} \Di_\lambda H_n(y)dy\le V_H(O).
\end{align}
Fix $n\in\NN$ and $x\in [\Di_\lambda H\le n]$. Then for every $\rho\ssup 0$ we have
\begin{align}\label{inequality-rn}
\inf_{\Q\in \mathfrak{B}_{x,\rho}}\frac{H(\Q)}{\lambda(\Q)}\le n.
\end{align} 
For each $n\in\NN$ and each $\rho\ssup0$ we set $A_n:=\{\Q\in{\mathfrak{B}_{x,\rho}}:H(\Q)\le n\lambda(\Q)\}$ and $B_n:=\mathfrak{B}_{x,\rho}\setminus A_n$. Then 
\begin{align*}
\Di_\lambda H_n(x)&=\sup_{\rho>0}\inf_{\Q\in \mathfrak{B}_{x,\rho}}\frac{H_n(\Q)}{\lambda(\Q)}\\
&=\sup_{\rho>0}\min\left\{\inf_{\Q\in A_n}\frac{H_n(\Q)}{\lambda(\Q)},\inf_{\Q\in B_n}\frac{H_n(\Q)}{\lambda(\Q)}\right\}\\
&=\sup_{\rho>0}\min\left\{\inf_{\Q\in A_n}\frac{H(\Q)}{\lambda(\Q)},n\right\}\\
&\ge\sup_{\rho>0}\min\left\{\inf_{\Q\in \mathfrak{B}_{x,\rho}}\frac{H(\Q)}{\lambda(\Q)},n\right\}.
\end{align*}
Using~\eqref{inequality-rn} we find
\[
\Di_\lambda H_n(x)\ge \sup_{\rho>0}\inf_{\Q\in \mathfrak{B}_{x,\rho}}\frac{H(\Q)}{\lambda(\Q)}=\Di_\lambda H(x).
\]
It follows that $\sup_{n\in\NN} \Di_\lambda H_n(x)= \Di_\lambda H(x)$ for all $x\in O$ and thus~\eqref{monotone-VH} becomes 
\begin{align*}
\int_O \Di_\lambda H(y)dy\le V_H(O).
\end{align*}
The proof is complete.\quad\mbox{\qedsymbol}

\def\epsn{{\eps_n}}

\section{Proof of main results}\label{Proof-main-theorem}
%%%%%%%%%%%%%%%%%%
\subsection{Proof of Lemma~\ref{lemma-H-theorem}} Fix $O\in\O(\Omega)$ and $u\in{\rm M}_\FF(O)$. By Theorem~\ref{main-result-eps}~\ref{Lower Bound} there exists $\{u_{\eps_n}\}_{\eps_n}$ with $\sup_n F_{\eps_n}(u_{\eps_n};\Omega)\sinf\infty$ such that $u_{\eps_n}\to u$ in $L^p(\Omega;\RR^m)$ as $n\to \infty$ and
\begin{align*}
\infty\ssup\FF_-^\D(u;O)\ge\FF_-(u;O)\ge \int_O\lim_{\rho\to 0} \lim_{n\to \infty}\frac{F_{\eps_n}(u_{\eps_n};\Q_\rho(x))}{\rho^d}dx.
\end{align*}
since Remark~\ref{domain-M}. It follows that for almost all $x\in O$
\begin{align}\label{ineq-H2}
\liminf_{\rho\to 0} \frac{\FF_-(u;\Q_\rho(x))}{\rho^d}\ge \lim_{\rho\to 0} \lim_{n\to \infty}\frac{F_{\eps_n}(u_{\eps_n};\Q_\rho(x))}{\rho^d}.
\end{align}
Using the local inequalities~\ref{H0},~\ref{H1}, ~\ref{H2},~\eqref{ineq-H2} and Theorem~\ref{main-result-eps}~\ref{upper bound} we have for almost all $x\in O$ 
\begin{align}
\limsup_{\rho\to 0}\frac{\FF_+(u;\Q_\rho(x))}{\rho^d}&\le \limsup_{\rho\to 0}\frac{\FF_+^\D(u;\Q_\rho(x))}{\rho^d} \notag\\
&\le \Di_\lambda \M_+(u;\cdot)(x)\notag\\
&\le \liminf_{\rho\to 0}\frac{\M_+(u;\Q_\rho(x))}{\rho^d}\notag\\
%&\le  \liminf_{\rho\to 0}\frac{\M_+(u_x;\Q_\rho(x))}{\rho^d}\notag\\
&\le \limsup_{\rho\to 0}\frac{\M_+(u_x;\Q_\rho(x))}{\rho^d}\notag\\
&\le \liminf_{\rho\to 0}\frac{\M_-(u_x;\Q_\rho(x))}{\rho^d}\notag\\
&\le \lim_{\rho\to 0}\lim_{n\to \infty}\frac{F_{\eps_n}(u_{\eps_n};\Q_\rho(x))}{\rho^d}\notag\\
&\le \liminf_{\rho\to 0} \frac{\FF_-(u;\Q_\rho(x))}{\rho^d}.\label{liminf-choice}
\end{align}
From the last inequality~\eqref{liminf-choice} we have the following inequalities
\begin{align*}
 \liminf_{\rho\to 0} \frac{\FF_-(u;\Q_\rho(x))}{\rho^d}&\le \liminf_{\rho\to 0} \frac{\FF_+(u;\Q_\rho(x))}{\rho^d}\le \limsup_{\rho\to 0} \frac{\FF_+(u;\Q_\rho(x))}{\rho^d}\\
 &\mbox{ and }\\
  \liminf_{\rho\to 0} \frac{\FF_-(u;\Q_\rho(x))}{\rho^d}&\le \limsup_{\rho\to 0} \frac{\FF_-(u;\Q_\rho(x))}{\rho^d}\le \limsup_{\rho\to 0} \frac{\FF_+(u;\Q_\rho(x))}{\rho^d}\\
  &\mbox{ and }\\
 \liminf_{\rho\to 0} \frac{\FF_-(u;\Q_\rho(x))}{\rho^d}&\le  \liminf_{\rho\to 0} \frac{\FF_+^\D(u;\Q_\rho(x))}{\rho^d}\le \limsup_{\rho\to 0} \frac{\FF_+^\D(u;\Q_\rho(x))}{\rho^d}\\
& \mbox{ and }\\
\liminf_{\rho\to 0} \frac{\FF_-(u;\Q_\rho(x))}{\rho^d}&\le  \liminf_{\rho\to 0} \frac{\FF_-^\D(u;\Q_\rho(x))}{\rho^d}\le \limsup_{\rho\to 0} \frac{\FF_-^\D(u;\Q_\rho(x))}{\rho^d}\\&\le \limsup_{\rho\to 0} \frac{\FF_+^\D(u;\Q_\rho(x))}{\rho^d}
\end{align*}
for all $x\in O$. It follows that for almost all $x\in O$
\begin{align*}
\lim_{\rho\to 0} \frac{\FF_-(u;\Q_\rho(x))}{\rho^d}&=\lim_{\rho\to 0} \frac{\FF_+(u;\Q_\rho(x))}{\rho^d}\\
&=\lim_{\rho\to 0} \frac{\FF_+^\D(u;\Q_\rho(x))}{\rho^d}=\lim_{\rho\to 0} \frac{\FF_-^\D(u;\Q_\rho(x))}{\rho^d}\\
&=\Di_\lambda \M_+(u;\cdot)(x)=\liminf_{\rho\to 0}\frac{\M_+(u;\Q_\rho(x))}{\rho^d}\\
&=\limsup_{\rho\to 0}\frac{\M_+(u_x;\Q_\rho(x))}{\rho^d}=\liminf_{\rho\to 0}\frac{\M_-(u_x;\Q_\rho(x))}{\rho^d}.
\end{align*}
So, the proof is complete since $O\ni x\mapsto \Di_\lambda \M_+(u;\cdot)(x)$ is measurable by Lemma~\ref{measurability}.\quad\mbox{\qedsymbol}

\subsection{Proof of Theorem~\ref{main-result-eps}}%~\ref{Lower Bound}}
\subsubsection*{Proof of Theorem~\ref{main-result-eps}~\ref{Lower Bound}}
Let $(u,O)\in W^{1,p}(\Omega;\RR^m)\times\O(\Omega)$ be such that $\FF_-(u;O)\sinf\infty$. There exists a sequence $\{u_\epsn\}_n\subset W^{1,p}(\Omega;\RR^m)$ such that
\begin{align}\label{bounds-step3}
u_\epsn\to u\mbox{ in }L^p(\Omega;\RR^m),\quad\lim_{n\to \infty} F_\epsn(u_\epsn;O)=\FF_-(u;O)\;\mbox{ and }\;\sup_{n} F_\epsn(u_\epsn;O)\sinf\infty.
\end{align}
By~\ref{C1}, for each $\eps\ssup0$ we consider the Borel measure  $\nu_\eps$ whose the trace on $\O(\Omega)$ is $F_\eps(u_\eps;\cdot)$. From the last inequality of~\eqref{bounds-step3} we can rewrite that the sequence of Borel measures $\left\{\mu_n:=\nu_\epsn\lfloor_{O}\right\}_n$ satisfies $\sup_{n}\mu_n(O)\sinf\infty$. So, there exists a Borel measure $\mu$ on $O$ such that (up to a subsequence) $\mu_n\stackrel{\ast}{\wto}\mu$. By Lebesgue decomposition theorem, we have $\mu=\mu_a+\mu_s$ where $\mu_a$ and $\mu_s$ are nonnegative Borel measures such that $\mu_a\ll \lambda\lfloor_O$ and $\mu_s\perp \lambda\lfloor_O$, and from Radon-Nikodym theorem we deduce that there exists $f\in L^1(O;\RR^+)$, given by
\begin{align*}
f(x)=\lim_{\rho\to 0}{\mu_a(\Q_\rho(x))\over\rho^d}=\lim_{\rho\to 0}{\mu(\Q_\rho(x))\over\rho^d}\;\;\hbox{ a.e. in }O 
\end{align*}
with $\Q_\rho(x):=x+\rho Y$, such that
\[
\mu_a(A)=\int_A f(x)dx\hbox{ for all measurable sets }A\subset O.
\]
By Alexandrov theorem we see that 
\begin{align*}
\FF_-(u;O)&=\lim_{n\to \infty} F_\epsn(u_\epsn;O)\\&=\lim_{n\to \infty}\mu_n(O)
\geq\mu(O)=\mu_a(O)+\mu_s(O)\geq\mu_a(O)=\int_O f(x)dx,
\end{align*}
and 
\begin{align*}
f(x)=\lim_{\rho\to 0}\lim_{n\to \infty} \frac{\mu_n(\Q_\rho(x))}{\rho^d}=\lim_{\rho\to 0}\lim_{n\to \infty} \frac{F_\epsn(u_\epsn;\Q_\rho(x))}{\rho^d}\;\;\hbox{ a.e. in }O.\quad\mbox{\qedsymbol}
\end{align*}

\bigskip

\subsubsection*{Proof of Theorem~\ref{main-result-eps}~\ref{upper bound}} For each $u\in W^{1,p}(\Omega;\RR^m)$ we denote by $\Mm_+(u;\cdot):\O(\Omega)\to [0,\infty]$ the Vitali envelope of $\M_+(u;\cdot)$, i.e.,
\begin{align*}
\Mm_+(u;O):=\Mm_{\M_+(u;\cdot)}(O).
\end{align*}
The proof consists to show that for every $O\in\O(\Omega)$ and every $u\in {\rm M}_\FF(O)$ the following inequality holds
\begin{align}\label{ineq-racine}
\FF_+^\D(u;O)\le \Mm_+(u;O).
\end{align}
Indeed, using Proposition~\ref{inequality-Vitali-envelope} we obtain
\begin{align*}
\FF_+(u;O)\le \FF_+^\D(u;O)\le \Mm_+(u;O)&=\int_O\Di_\lambda \M_+(u;\cdot)(x)dx.
%&=\int_O\lim_{\rho\to 0}\frac{\M_+(u;\Q_\rho(x))}{\rho^d}dx.
\end{align*}
Let us prove~\eqref{ineq-racine} now. Fix $O\in\O(\Omega)$ and $u\in {\rm M}_\FF(O)$. Note that by Remarks~\ref{trivial-vitali} we have for some $\mu_u\in\mesabs(O)$
\begin{align}\label{eq0: step1}
\Mm_+(u;O)\le \mu_u(O)\sinf\infty.
\end{align}
Fix $\eps\in ]0,1[$. Choose $\left\{\overline{\Q}_i\right\}_{i\in I}\in\mathcal{V}_{\eps}(O)$ such that
\begin{align}\label{eq1: step1}
\sum_{i\in I}\M_+(u;\Q_i)\le \Mm^{\eps}_+(u;O)+\frac\eps2\le \Mm_+(u;O)+\frac\eps2.
\end{align}
Fix $\delta\in ]0,1[$. Given any $i\in I$, by definition of $\M_\delta(u;\Q_i)$, there exists $v_i\in u+{W}^{1,p}_0(\Q_i;\RR^m)$ such that 
\begin{align}\label{eq2: step1}
F_\delta(v_i;\Q_i)\le \M_\delta(u;\Q_i)+\frac\delta2 \frac{\lambda(\Q_i)}{\lambda(O)}.
\end{align}
Define $u_{\delta,\eps}\in u+W^{1,p}_0(O;\RR^m)$ by 
\[
u_{\delta,\eps}:=\sum_{i\in I} v_i\mathds{1}_{\Q_i}+u\mathds{1}_{\displaystyle\Omega\setminus\mathop{\cup}_{i\in I}\Q_i}.
\] 
Using \ref{C1} and \ref{C2} we have from \eqref{eq2: step1}
\begin{align*}
F_\delta(u_{\delta,\eps};O)&= \sum_{i\in I}F_\delta(v_i;\Q_i)+F_\delta(u;{O\setminus\mathop{\cup}_{i\in I}\Q_i})\\
&=\sum_{i\in I}F_\delta(v_i;\Q_i)\notag\\
&\le \sum_{i\in I}\M_\delta(u;\Q_i)+\frac\delta2.\notag
\end{align*}
Since $u\in {\rm M}_\FF(O)$ there exists $\mu_u\in\mesabs(O)$ such that $\sup_{\delta\in ]0,1]} \M_\delta(u;U)\le\mu_u(U)$ for all open set $U\subset O$. For every $\eta\ssup 0$ there exists a finite set $I_\eta\subset I$ such that $\mu_u(O\setminus \cup_{i\in I_\eta}\Q_i)\le\eta$. It follows that $\sum_{i\in I\setminus I_\eta} \M_\delta(u;\Q_i)\le \eta$. Hence, for any $\eta\ssup 0$
\begin{align}\label{eq4: step1}
\limsup_{\delta\to 0}\sum_{i\in I}\M_\delta(u;\Q_i)&\le \limsup_{\delta\to 0}\sum_{i\in I_\eta}\M_\delta(u;\Q_i)+\limsup_{\delta\to 0}\sum_{i\in I\setminus I_\eta}\M_\delta(u;\Q_i)\\
&\le \sum_{i\in I}\M_+(u;\Q_i)+\eta.\notag
\end{align}
Therefore collecting \eqref{eq1: step1}, \eqref{eq4: step1}, and passing to the limit $\eps\to 0$, we have
\begin{align}\label{eq7: step1}
\limsup_{\eps\to 0}\limsup_{\delta\to 0}F_\delta(u_{\delta,\eps};O)\le \Mm_+(u;O).
\end{align}
From the $p$-coercivity \ref{C0}, \eqref{eq7: step1} and \eqref{eq0: step1}, we deduce
\begin{align}\label{eq3: step1}
\limsup_{\eps\to 0}\limsup_{\delta\to 0}\int_O \vert\nabla u_{\delta,\eps}\vert^pdx\sinf\infty.
\end{align}
By Poincar\'e inequality there exists $K\ssup 0$ depending only on $p$ and $d$ such that for each $v_i\in u+{W}^{1,p}_0(\Q_i;\RR^m)$
\begin{align*}
\int_{\Q_i}\vert v_i-u\vert^p dx\le K\eps^p\int_{\Q_i}\vert \nabla v_i-\nabla u\vert^pdx
\end{align*}
since $\diam(\Q_i)\sinf\eps$. Summing over $i\in I$ 
we obtain
\begin{align*}
\int_O\vert u_{\delta,\eps}-u\vert^p dx\le 2^{p-1}K\eps^p(\int_{O}\vert \nabla u_{\delta,\eps}\vert^p dx+\int_O \vert\nabla u\vert^pdx)
\end{align*}
which shows, by using~\eqref{eq3: step1}, that 
\begin{align}\label{eq9: step1}
\limsup_{\eps\to 0}\limsup_{\delta\to 0}\int_\Omega\vert u_{\delta,\eps}-u\vert^p dx=0.
\end{align}
A simultaneous diagonalization of~\eqref{eq7: step1} and~\eqref{eq9: step1} gives a sequence $\{u_\delta:=u_{\delta,\eps(\delta)}\}_\delta\subset u+W^{1,p}(O;\RR^m)$ such that $u_\delta\to u$ in $L^p(\Omega;\RR^m)$ and 
\begin{align*}
\FF_+^\D(u;O)\le \limsup_{\delta\to 0}F_\delta(u_{\delta};O)\le \Mm_+(u;O)
\end{align*}
by the definition of $\FF_+^\D(u;O)$. The proof is complete.\quad\mbox{\qedsymbol}

%%%%%Section-5%%%%
\section{Applications}\label{Integral representation with $p$-growth conditions}
\subsection{General $\Gammap$-convergence result in the $p$-growth case} 
For each $\eps\in ]0,1]$ we consider a family of functionals $\FF:=\{F_\eps\}_{\eps\in ]0,1]},\;F_\eps:W^{1,p}(\Omega;\RR^m)\times\O(\Omega)\to [0,\infty]$. 

Consider the following condition:
\begin{hypG}
\item\label{L1} there exist $\beta\ssup 0$ and $\nu$ a nonnegative finite Borel measure on $\Omega$ absolutely continuous with respect to the Lebesgue measure such that for every $(V,u,\eps)\in\O(\Omega)\times W^{1,p}(\Omega;\RR^m)\times]0,1]$ we have
\begin{align*}
 \frac{\M_\eps(u;V)}{\vert V\vert}\le \beta(\frac{\nu(V)}{\vert V\vert}+\fint_V\vert u\vert^pdx+\fint_V\vert \nabla u\vert^pdx)
\end{align*}
\end{hypG}

The following result can be seen as a nonconvex extension of Theorem IV of \cite[p. 265]{dalmaso-modica86}. Indeed, if for each $\eps\ssup0$ we set $F_\eps:W^{1,p}(\Omega;\RR^m)\times \O(\Omega)\to [0,\infty]$ defined by
\begin{align*}
F_\eps(u;O):=\int_O \f_\eps(x,u(x),\nabla u(x))dx
\end{align*}
where $\f_\eps:\Omega\times\RR^d\times\MM^{m\times d}\to [0,\infty[$ is a Borel measurable function with $p$-growth and $p$-coercivity, i.e., 
\begin{align}\label{p-growth-p-coercivity}
\exists \alpha\ssup0\quad\exists\beta\ssup0\quad&\exists a\in L^1(\Omega)\quad \forall (x,v,\xi)\in\Omega\times\RR^d\times\MM^{m\times d}\quad  \forall\eps\ssup0\\ \notag \\
&\alpha\vert\xi\vert^p\le \f_\eps(x,v,\xi)\le\beta(a(x)+\vert v\vert^p+\vert\xi\vert^p)\notag
\end{align}
then~\ref{L1} holds with $\nu=a(\cdot)\lambda$ and $\FF=\{F_\eps\}_\eps\subset \mathcal{I}(p,\alpha)$.
\begin{theorem}\label{p-growth-theorem} Assume that $\FF\subset \mathcal{I}(p,\alpha)$. Let $u\in W^{1,p}(\Omega;\RR^m)$ and $O\in\O(\Omega)$. If~\ref{H0} and~\ref{L1} hold then the family $\FF(\cdot;O)$ $\Gammap$-converges at $u$ to 
\begin{align*}
\FF_0(u,O):=\int_O L_{0}(x,u(x),\nabla u(x))dx
\end{align*}
where $\f_0(\cdot,u(\cdot),\nabla u(\cdot))$ is given by~\eqref{generic-formula}.
\end{theorem}
\begin{proof} Since \ref{L1} we see that  
\begin{align*}
{\rm M}_\FF(O)=W^{1,p}(\Omega;\RR^m).
\end{align*}
Fix $u\in W^{1,p}(\Omega;\RR^m)$. Following Theorem~\ref{H-theorem} it is enough to show that~\ref{H1} and~\ref{H2} hold.

\medskip

We begin by showing~\ref{H2}. Fix $x\in O$ such that
\begin{align}
&\lim_{r\to 0}\fint_{\Q_r(x)}\vert  u\vert^pdy=\vert u(x)\vert^p\sinf\infty;\label{negli0}\\
&\lim_{r\to 0}\fint_{\Q_r(x)}\vert \nabla u\vert^pdy=\vert \nabla u(x)\vert^p\sinf\infty;\label{negli1}\\
&\lim_{r\to 0}\frac{1}{r^p}\fint_{\Q_r(x)}\vert u_x-u\vert^pdy=0;\label{negli2}\\
&\lim_{r\to 0}\frac{\nu(\Q_r(x))}{r^d}=D_\lambda \nu(x)\sinf\infty.\label{negli4}
\end{align}
Fix $\eps\ssup 0, s\in ]0,1[$ and $\rho\ssup 0$. Let $\phi\in W^{1,\infty}_0(\Q_\rho(x);[0,1])$ be a cut-off function between $\overline{\Q}_{s\rho}(x)$ and ${\Q}_{\rho}(x)$ (i.e., $\phi=1$ on $\overline{\Q}_{s\rho}(x)$ and $\phi=0$ on $O\setminus \Q_{s\rho}(x)$) such that
\begin{align*}
\Vert \nabla\phi\Vert_{L^\infty({\Q}_{\rho}(x))}\le \frac{4}{\rho(1-s)}.
\end{align*}
Let $v_\eps\in u_x+W^{1,p}_0(\Q_{s\rho}(x);\RR^m)$ be such that
\begin{align}\label{eps-min-rec1}
F_\eps(v_\eps;\Q_{s\rho}(x))\le\eps(s\rho)^d+\M_\eps(u_x; \Q_{s\rho}(x)).
\end{align}
Set $w:=\phi v_\eps+(1-\phi)u$, we have $w\in u+W^{1,p}_0(\Q_\rho(x);\RR^m)$ and
\[
\nabla w:=\left\{
\begin{array}{cl}
 \nabla v_\eps& \text{ in } \Q_{s\rho}(x)  \\ \\
 \phi\nabla u(x)+(1-\phi)\nabla u+\nabla \phi\otimes(u_x-u) &      \text{ in } \Sigma_\rho(x)
\end{array}
\right.
\]
where $\Sigma_\rho(x):=\Q_{\rho}(x)\setminus\overline{\Q}_{s\rho}(x)$. 
We have 
\begin{align*}
\M_{\eps}(u;\Q_{\rho}(x))&= \M_\eps(w;\Q_{\rho}(x))\\
&\le \M_{\eps}(w; \Q_{s\rho}(x))+\M_\eps(w; \Sigma_\rho(x))\\
&\le F_\eps(v_\eps;\Q_{s\rho}(x))+\M_\eps(w; \Sigma_\rho(x))\\
&\le \eps(s\rho)^d+\M_\eps(u_x; \Q_{s\rho}(x))+\M_\eps(w; \Sigma_\rho(x))
\end{align*}
since Lemma~\ref{lemma-sub-sup-m} and~\eqref{eps-min-rec1}.
It follows that
\begin{align}\label{cut-off-ineq}
\frac{\M_{\eps}(u;\Q_{\rho}(x))}{\rho^d}\le \eps s^d+ s^d\frac{\M_{\eps}(u_x; \Q_{s\rho}(x))}{(s\rho)^d}+\frac{\M_\eps(w; \Sigma_\rho(x))}{\rho^d}.
\end{align}
We claim that~\ref{H2} is proved if 
\begin{align}\label{rest-local0}
\limsup_{s\to 1}\limsup_{\rho\to 0}\limsup_{\eps\to 0}\frac{\M_\eps(w; \Sigma_\rho(x))}{\rho^d}=0.
\end{align}
Indeed, passing to the limits $\eps\to 0$, $\rho\to 0$, $s\to 1$ in~\eqref{cut-off-ineq} we have
\begin{align}\label{rest-local10}
\liminf_{\rho\to 0}\frac{\M_{+}(u;\Q_{\rho}(x))}{\rho^d}&\le \liminf_{s\to 1^-}\liminf_{\rho\to 0}\frac{\M_{+}(u_x;\Q_{s\rho}(x))}{(s\rho)^d}
\\&\le\limsup_{\rho\to 0}\frac{\M_{+}(u_x;\Q_{\rho}(x))}{\rho^d}.\notag
\end{align}

\bigskip

So, it remains to prove~\eqref{rest-local0}. Using \ref{L1} we have for some $C\ssup 0$ dependent on $p$ only
\begin{align*}
&\frac{\M_\eps(w; \Sigma_\rho(x))}{\rho^d}\\
&\le 
\beta(\frac{\nu( \Sigma_\rho(x))}{\rho^d}+\frac{1}{\rho^d}\int_{\Sigma_\rho(x)} \left\vert \phi\nabla u(x)+(1-\phi)\nabla u+\nabla \phi\otimes(u_x-u)\right\vert^pdy)\\
&+\frac{\beta}{\rho^d}\int_{\Sigma_\rho(x)} \left\vert \phi u_x+(1-\phi)u\right\vert^pdy\\
&\le C\beta(\frac{\nu( \Sigma_\rho(x))}{\rho^d}+(1-s^d)\vert \nabla u(x)\vert^p+\fint_{\Q_{\rho}(x)}\vert \nabla u\vert^pdy-s^d\fint_{\Q_{s\rho}(x)}\vert \nabla u\vert^pdy)\\
&+C\beta(\frac{4^p}{(1-s)^p}(\frac{1}{\rho^p}\fint_{\Q_{\rho}(x)}\vert u_x-u\vert^pdy-\frac{s^{d+p}}{(s\rho)^p}\fint_{\Q_{s\rho}(x)}\vert u_x-u\vert^pdy))\\
&+C\beta\rho^p(\frac{1}{\rho^p}\fint_{\Q_\rho(x)}\vert u_x-u\vert^pdy-\frac{s^{d+p}}{(s\rho)^p}\fint_{\Q_{s\rho}(x)}\vert u_x-u\vert^pdy)\\
&+C\beta(\fint_{\Q_\rho(x)}\vert u\vert^p-s^d \fint_{\Q_{s\rho}(x)}\vert u\vert^p).
\end{align*}
Taking~\eqref{negli0},~\eqref{negli1},~\eqref{negli2} and \eqref{negli4} into account and passing to the limits $\eps\to 0$ then $\rho\to 0$ we obtain 
\begin{align*}
\limsup_{\rho\to 0}\limsup_{\eps\to 0}\frac{\M_\eps(w; \Sigma_\rho(x))}{\rho^d}\le C\beta(1-s^d)(D_\lambda \nu(x)+\vert u(x)\vert^p+\vert \nabla u(x)\vert^p).
\end{align*}
\bigskip
Letting $s\to 1$ we obtain~\eqref{rest-local0}.

Let us prove~\ref{H1} now. Consider a sequence $\{\varphi_\eps\}_\eps\subset W^{1,p}(\Omega;\RR^m)$ such that $\varphi_{\eps}\to 0$ in $L^p(\Omega;\RR^m)$ as $\eps\to 0$ and satisfying $\sup_{\eps>0} F_{\eps}(u+\varphi_\eps;\Omega)\sinf\infty$. Set $\mu_\eps(\cdot):=F_\eps(u+\varphi_\eps;\cdot)$ for any $\eps\ssup 0$. There exists a subsequence (not relabeled) and a nonnegative Radon measure $\mu_0$ such that 
\begin{align}\label{weak-convergence}
\mu_\eps\stackrel{\ast}{\wto}\mu_0.
\end{align}
Fix $\eps\ssup 0, s\in ]1,2[$ and $\rho\ssup 0$. Fix $x\in O$ such that~\eqref{negli0},~\eqref{negli1},~\eqref{negli2} and~\eqref{negli4} hold and
\begin{align}
D_\lambda \mu_0(x):=\lim_{r\to 0}\frac{\mu_0(\Q_r(x))}{r^d}\sinf\infty.\label{negli1prime}
\end{align}
Let $\phi\in W^{1,\infty}_0(\Q_{s\rho}(x);[0,1])$ be a cut-off function between $\overline{\Q}_{\rho}(x)$ and ${\Q}_{s\rho}(x)$ such that
\begin{align*}
\Vert \nabla\phi\Vert_{L^\infty({\Q}_{s\rho}(x))}\le \frac{4}{\rho(s-1)}.
\end{align*}
Let $v_\eps\in (u+\varphi_\eps)+W^{1,p}_0(\Q_{\rho}(x);\RR^m)$ be such that
\begin{align}\label{eps-min-rec2}
F_\eps(v_\eps;\Q_{\rho}(x))\le\eps\rho^d+\M_\eps(u+\varphi_\eps; \Q_{\rho}(x)).
\end{align}
Set $w:=\phi v_\eps+(1-\phi)u_x$, we have $w\in u_x+W^{1,p}_0(\Q_{s\rho}(x);\RR^m)$ and
\begin{align*}
\nabla w:=\left\{
\begin{array}{cl}
 \nabla v_\eps& \text{ in } \Q_{\rho}(x)  \\ \\
 \phi(\nabla u+\nabla \varphi_\eps)+(1-\phi)\nabla u(x)+\nabla \phi\otimes(u+\varphi_\eps-u_x) &      \text{ in }\Sigma_{\rho}(x)
\end{array}
\right.
\end{align*}
where $\Sigma_{\rho}(x):=\Q_{s\rho}(x)\setminus\overline{\Q}_\rho(x)$. We have
\begin{align}\label{cut-off-ineq-prime}
s^d\frac{\M_{\eps}(u_x;\Q_{s\rho}(x))}{(s\rho)^d}&= s^d\frac{\M_\eps(w;\Q_{s\rho}(x))}{(s\rho)^d}\\
&\le \frac{\M_\eps(w;\Q_{\rho}(x))}{\rho^d}+ \frac{\M_\eps(w;\Sigma_{\rho}(x))}{\rho^d}\notag\\
&\le \frac{F_\eps(v_\eps;\Q_{\rho}(x))}{\rho^d}+\frac{\M_\eps(w;\Sigma_{\rho}(x))}{\rho^d}\notag\\
&\le \eps+\frac{\M_\eps(u+\varphi_\eps; \Q_{\rho}(x))}{\rho^d}+\frac{\M_\eps(w;\Sigma_{\rho}(x))}{\rho^d}\notag\\
&\le\eps+\frac{F_\eps(u+\varphi_\eps; \Q_{\rho}(x))}{\rho^d}+\frac{\M_\eps(w;\Sigma_{\rho}(x))}{\rho^d}\notag
\end{align}
since Lemma~\ref{lemma-sub-sup-m} and~\eqref{eps-min-rec2}. We claim that~\ref{H1} is proved if 
\begin{align}\label{rest-local1}
\limsup_{s\to 1}\limsup_{\rho\to 0}\limsup_{\eps\to 0}\frac{\M_\eps(w; \Sigma_\rho(x))}{\rho^d}=0.
\end{align}
Indeed, passing to the limits $\eps\to 0$, $\rho\to 0$, $s\to 1$ in~\eqref{cut-off-ineq-prime} %and using Corollary~\ref{min-sub} 
we have
\begin{align}\label{rest-local11}
\limsup_{\rho\to 0}\liminf_{\eps\to 0}\frac{F_\eps(u+\varphi_\eps;\Q_{\rho}(x))}{\rho^d}&\ge \limsup_{s\to 1}\limsup_{\rho\to 0}\frac{\M_{-}(u_x;\Q_{s\rho}(x))}{(s\rho)^d}\\
&\ge\liminf_{\rho\to 0}\frac{\M_{-}(u_x;\Q_{\rho}(x))}{\rho^d}.\notag
\end{align}

\bigskip

So, it remains to prove~\eqref{rest-local1}. Using \ref{L1} we have for some $C\ssup 0$ dependent on $p$ only
\begin{align}
&\frac{\M_\eps(w;\Sigma_{\rho}(x))}{\rho^d}\label{eq-rest-prime}\\
&\le \beta\frac{1}{\rho^d}\int_{\Sigma_\rho(x)} \left\vert \phi (u+\varphi_\eps)+(1-\phi)u_x\right\vert^p dy\notag\\
&+ C\beta((s^d-1)\vert\nabla u(x)\vert^p+\frac{\nu(\Sigma_{\rho}(x))}{\rho^d})\notag\\
&+C\beta(\frac{1}{\rho^d}\int_{\Sigma_\rho(x)} \left\vert \nabla u+\nabla\varphi_\eps\right\vert^p +\frac{1}{\rho^d}\int_{\Sigma_\rho(x)}\left\vert\nabla \phi\otimes(u+\varphi_\eps-u_x)\right\vert^pdy)\notag\\
&\le C\beta((s^d-1)\vert\nabla u(x)\vert^p+\frac{\nu(\Sigma_{\rho}(x))}{\rho^d}+\frac{1}{\alpha}\frac{1}{\rho^d}F_\eps(u+\varphi_\eps;{\Sigma_\rho(x)}))\notag\\
&+C\beta\frac{2^{3p-1}}{(s-1)^p}
(s^{d+p}\frac{1}{(s\rho)^p}\fint_{\Q_{s\rho}(x)}\left\vert u_x-u\right\vert^pdy-\frac{1}{\rho^p}\fint_{\Q_{\rho}(x)}\left\vert u_x-u\right\vert^pdy)\notag\\
&+C\beta\frac{2^{3p-1}}{(s-1)^p}
(s^{d+p}\frac{1}{(s\rho)^p}\fint_{\Q_{s\rho}(x)}\left\vert \varphi_\eps\right\vert^pdy-\frac{1}{\rho^p}\fint_{\Q_{\rho}(x)}\left\vert \varphi_\eps\right\vert^pdy)\notag\\
&+C\beta\rho^p(\frac{1}{\rho^p}\fint_{\Q_\rho(x)}\vert u_x-u\vert^pdy-\frac{s^{d+p}}{(s\rho)^p}\fint_{\Q_{s\rho}(x)}\vert u_x-u\vert^pdy)\notag\\
&+C\beta(\fint_{\Q_\rho(x)}\vert u\vert^p-s^d \fint_{\Q_{s\rho}(x)}\vert u\vert^p)+C\beta\frac{1}{\rho^d}\int_{\Sigma_\rho(x)}\left\vert \varphi_\eps\right\vert^pdy.\notag
\end{align}
Using \eqref{weak-convergence} and Alexandrov theorem we have
\begin{align*}
\limsup_{\eps\to 0}\frac{1}{\rho^d}F_\eps(u+\varphi_\eps;{\Sigma_\rho(x)})&=\limsup_{\eps\to 0}\frac{\mu_\eps({\Sigma_\rho(x)})}{\rho^d}\\
&\le \limsup_{\eps\to 0}\frac{\mu_\eps({\overline{\Sigma}_\rho(x)})}{\rho^d}\\
&\le \frac{\mu_0({\overline{\Sigma}_\rho(x)})}{\rho^d}\\
&\le s^d \frac{\mu_0(\overline{\Q}_{s\rho}(x))}{(s\rho)^d}-\frac{\mu_0(\Q_{\rho}(x))}{\rho^d}.
\end{align*}
Letting $\rho\to 0$ we deduce by using~\eqref{negli1prime}
\begin{align}\label{eq-intermediate}
\limsup_{\rho\to 0}\limsup_{\eps\to 0}\frac{1}{\rho^d}F_\eps(u+\varphi_\eps;{\Sigma_\rho(x)})\le (s^d-1)D_\lambda \mu_0(x).
\end{align}
Taking~\eqref{negli0},~\eqref{negli1},~\eqref{negli2},~\eqref{negli4} and~\eqref{eq-intermediate} into account and passing to the limits $\eps\to 0$ then $\rho\to 0$ in~\eqref{eq-rest-prime} we obtain
\begin{align*}
\limsup_{\rho\to 0}\limsup_{\eps\to 0}\frac{\M_\eps(w;\Sigma_{\rho}(x))}{\rho^d}\le C\beta(s^d-1)(D_\lambda\nu(x)+\vert u(x)\vert^p+\vert\nabla u(x)\vert^p+D_\lambda \mu_0(x))
\end{align*}
since $\varphi_{\eps}\to 0$ in $L^p(\Omega;\RR^m)$ as $\eps\to 0$. Passing to the limit $s\to 1$ we finally proved~\eqref{rest-local1}.\mbox{\qedhere}
\end{proof}

\bigskip

As an illustration of Theorem~\ref{p-growth-theorem} we give two elementary examples.
%\subsubsection{Elementary examples}
\begin{example}[Integrands ``almost" nondecreasing] For each $\eps\ssup0$ we consider $L_\eps:\MM^{m\times d}\to[0,\infty[$ a Borel measurable function such that
\begin{hypG}[resume]
\item\label{L2} $\displaystyle\exists \gamma\ge 0\quad\exists\delta\ssup0\quad\forall \eps\ssup0\quad\forall \eta\in ]0,\eps[\quad\forall (x,v,\xi)\in\Omega\times\RR^d\times\MM^{m\times d}$\[\f_\eps(\xi)\le \f_\eta(\xi)+\gamma\vert \eps-\eta\vert^\delta.\]
\end{hypG}
Note that if $\gamma=0$ then $\eps\mapsto \f_\eps(\cdot)$ is nondecreasing when $\eps$ is decreasing.

\medskip

We define $F_\eps:L^p(\Omega;\RR^m)\times\O(\Omega)\to [0,\infty]$ by 
\begin{align*}
F_\eps(u;O):=\int_O L_\eps(\nabla u(x))dx.
\end{align*}
Then it is easy to see that~\ref{H0} holds. If we assume~\eqref{p-growth-p-coercivity} then~\ref{L1} holds.
\end{example}
\begin{example}[Constant integrands with perturbation] Let $W:\Omega\times\RR^d\times\MM^{m\times d}\to [0,\infty[$ be a Borel measurable integrand satisfying $p$-growth and $p$-coercivity~\eqref{p-growth-p-coercivity} and~\ref{H0}. Let $\{\Phi_\eps\}_\eps\subset L^1(\Omega;\RR^+)$ such that 
\begin{hyp2}
\item\label{uniform-bound-l1-example} there exists $g\in L^1(\Omega)$ such that $\Phi_\eps(x)\le g(x)$ for all $x\in\Omega$ and all $\eps\ssup0$;\\

\item\label{uniform-bound-l1-example1} there exists a nonnegative Borel measure $\Phi_0$ such that
\begin{align}\label{weak-to-measure-example0}
\Phi_\eps(\cdot)\lambda\stackrel{\ast}{\wto}\Phi_0\quad\mbox{ as }\eps\to 0.
\end{align}
\end{hyp2}
For each $\eps\ssup0$ we set $\f_\eps:\Omega\times\RR^d\times\MM^{m\times d}\to [0,\infty[$ defined by
\begin{align*}
\f_\eps(x,v,\xi)=W(x,v,\xi)+\Phi_\eps(x).
\end{align*}
Then for each $O\in\O(\Omega)$ the family $\FF(\cdot;O)$ $\Gammap$-converges to 
\begin{align*}
\FF_0(u;O)=\int_O W_{0}(x,u(x),\nabla u(x))+D_\lambda\Phi_0(x)dx.
\end{align*}
Indeed, we have that \ref{L1} holds because of the $p$-growth of $W$ and~\ref{uniform-bound-l1-example}. Now, we have for almost all $x\in \Omega$
\begin{align*}
\lim_{\rho\to 0}\lim_{\eps\to 0}\fint_{\Q_\rho(x)}\Phi_\eps(y)dy=D_\lambda\Phi_0(x)
\end{align*}
since~\eqref{weak-to-measure-example0}. We can see that for every $x\in\Omega$, every $\eps\ssup0$, every $\rho\ssup0$ and every $u\in W^{1,p}(\Omega;\RR^m)$
\begin{align*}
&\inf\left\{\fint_{\Q_\rho(x)} W(y,w(y),\nabla w(y))dy:w\in u_x+W^{1,p}_0(\Q_\rho(x);\RR^m)\right\}\\
&+\fint_{\Q_\rho(x)}\Phi_\eps(y)dy.
\end{align*}
It means that {\rm (${\mathscr H}$)} holds and Theorem~\ref{abstract-homogen} applies with
\begin{align*}
\f_0(x,u(x),\nabla u(x))=W_{0}(x,u(x),\nabla u(x))+D_\lambda\Phi_0(x).
\end{align*}
We give a concrete example. Assume that $\Omega=B_1(0)\subset \RR^d$ the euclidean open ball with center $0$ and radius $1$. Let $g:\Omega\to [0,\infty]$ be defined by
\[
g(x):=\left\{
\begin{array}{cl}
 \frac{2}{\sqrt{\Vert x\Vert}} & \mbox{ if }x\in \Omega\setminus\{0\}      \\
 \infty &      \mbox{ if }x=0.
\end{array}
\right.
\]
where $\Vert\cdot\Vert$ is the euclidean norm. Then $g\in L^1(\Omega)$. For each $\eps\ssup0$ we set for every $x\in\Omega$
\begin{align*}
\Phi_\eps(x):=\frac{1}{\sqrt{\eps}}\mathds{1}_{B_\eps(0)}(x)+h(x)
\end{align*}
where $h\in L^1(\Omega)$ and satisfies $h(x)\le\frac12 g(x)$ for all $x\in\Omega$. Then~\ref{uniform-bound-l1-example} and ~\ref{uniform-bound-l1-example1} hold with 
\begin{align*}
\Phi_\eps(\cdot)\lambda\stackrel{\ast}{\wto}\Phi_0:=\delta_0+h\lambda\quad\mbox{ as }\eps\to 0
\end{align*}
where $\delta_0$ is the dirac measure at $0$.
It follows that 
\begin{align*}
D_\lambda\Phi_0(x)=h(x)\mbox{ a.e. in }\Omega.
\end{align*}
\end{example}
\subsection{Homogenization}\label{sub-homog}
Let $L:\RR^d\times\MM^{m\times d}\to[0,\infty[$ be a Borel measurable function which is $p$-coercive, i.e., there exists $\alpha\ssup0$ such that 
\begin{align*}
\alpha\vert\xi\vert^p\le \f(x,\xi)
\end{align*}
for all $(x,\xi)\in\Omega\times\MM^{m\times d}$. For each $\eps\ssup0$ we consider $F_\eps:W^{1,p}(\Omega;\RR^m)\times\O(\Omega)\to [0,\infty]$ given by 
\begin{align*}
F_\eps(u;O):=\int_O \f(\frac{x}{\eps},\nabla u(x))dx.
\end{align*}
The family $\FF=\{F_\eps\}_\eps\subset\I(p,\alpha)$. For each $\xi\in\MM^{m\times d}$ we define $\S_\xi:\O(\Omega)\to [0,\infty]$ a set function by 
\begin{align*}
\S_\xi^\f(O):=\inf\left\{\int_O \f(y,\xi+\nabla\varphi(y))dy:\varphi\in W^{1,p}_0(O;\RR^m)\right\}.
\end{align*}
\begin{definition} We say that $\f$ is an {\em ${\rm H}$-integrand} ({\rm H} stands for ``homogenizable") if

\bigskip

\begin{enumerate}
\item[{\rm (${\mathscr H}$)}]\label{L3} $\displaystyle\forall \xi\in\MM^{m\times d}\qquad
\limsup_{\rho\to 0}\limsup_{t\to\infty}\frac{\S_\xi^\f(t\Q_{\rho}(x))}{\lambda(t\Q_\rho(x))}=\liminf_{\rho\to 0}\liminf_{t\to\infty}\frac{\S_\xi^\f(t\Q_{\rho}(x))}{\lambda(t\Q_\rho(x))}\quad\mbox{a.e. in }\Omega.\\
$
\end{enumerate}
In this case we denote the common value by $\f_{\rm hom}(x,\xi)$. 
\end{definition}

We see that {\rm (${\mathscr H}$)} implies~\ref{H0}, indeed, for every $u\in{\rm M}_\FF(O)$ we have
\begin{align*}
\frac{\S_{\nabla u(x)}^\f(\frac{1}{\eps} \Q_\rho(x))}{\lambda(\frac{1}{\eps} \Q_\rho(x))}=\frac{\M_\eps(u_x\cdot;\Q_\rho(x))}{\rho^d}.
\end{align*}
for all $\eps\ssup0$ and all $x\in O$. So, we can deduce from Theorem~\ref{p-growth-theorem} the following result.
\begin{theorem}\label{abstract-homogen} If~\ref{L1} holds and $\f$ is an {\rm H}-integrand, i.e., {\rm (${\mathscr H}$)} holds. Then for each $O\in\O(\Omega)$ the family $\FF(\cdot;O)$ $\Gammap$-converges at every $u\in W^{1,p}(\Omega;\RR^m)$ to 
\begin{align*}
\FF_0(u,O)=\int_O \f_{\rm hom}(x,\nabla u(x))dx
\end{align*}
where 
\begin{align*}
\f_{\rm hom}(x,\xi)=\f_0(x,\xi)=\limsup_{\rho\to 0}\limsup_{t\to\infty}\frac{\S_\xi^\f(t\Q_{\rho}(x))}{\lambda(t\Q_\rho(x))}.
\end{align*}
for all $x\in O$ and $\xi\in\MM^{m\times d}$.
\end{theorem}
Theorem~\ref{abstract-homogen} becomes a ``classical" homogenization result when $\f_{\rm hom}$ does not depend on $x$. For instance, when $\f$ is $1$-periodic or almost periodic with respect to the first variable then by subadditive theorems \cite[Theorem 2.1 and Theorem 3.1]{licht-michaille02} the condition {\rm (${\mathscr H}$)} holds, i.e., $\f$ is an {\rm H}-integrand, and we have
\begin{align}
\f_{\rm hom}(\xi)&=\inf_{n\in\NN}\frac{\S_\xi^\f(nY)}{n^d} &&\quad\mbox{(periodic case)}\label{homog-formula-periodic}\\
\f_{\rm hom}(\xi)&=\lim_{n\to\infty}\frac{\S_\xi^\f(nY)}{n^d} &&\quad\mbox{(almost-periodic case)}.
\end{align}
\begin{example}[Periodic integrand with perturbation] Consider $W:\RR^d\times\MM^{m\times d}\to [0,\infty[$ be a Borel measurable function $1$-periodic with respect to the first variable, i.e., 
\begin{align*}
\forall x\in\RR^d\quad\forall z\in\ZZ^d\quad\forall \xi\in\MM^{m\times d}\qquad W(x+z,\xi)=W(x,\xi),
\end{align*}
and satisfying $p$-growth and $p$-coercivity, i.e, there exist $\alpha,\beta\ssup0$ such that 
\begin{align*}
\forall (x,\xi)\in\RR^d\times\MM^{m\times d}\qquad \alpha\vert \xi\vert^p\le W(x,\xi)\le \beta(1+\vert \xi\vert^p).
\end{align*}
Let $\Phi\in L^1_{\rm loc}(\RR^d;\RR^+)$ such that 
\begin{hyp2}
\item\label{uniform-bound-l1-example-bis} there exists $g\in L^1_{\rm loc}(\RR^d)$ such that $\Phi(\frac{x}{\eps})\le g(x)$ for all $x\in\Omega$ and all $\eps\ssup0$;\\

\item\label{uniform-bound-l1-example-bis1} there exists a nonnegative Borel measure $\Phi_0$ such that
\begin{align}\label{weak-to-measure-example}
\Phi(\frac{\cdot}{\eps})\lambda\stackrel{\ast}{\wto}\Phi_0\quad\mbox{ as }\eps\to 0.
\end{align}
\end{hyp2}
Let $\f:\RR^d\times\MM^{m\times d}\to [0,\infty[$ be defined by
\begin{align*}
\f(x,\xi)=W(x,\xi)+\Phi(x).
\end{align*}
Note that $\f$ is not periodic with respect to the first variable, because of the ``perturbation" $\Phi$.% which is not necessarily periodic.

\medskip

We consider the family $\FF=\{F_\eps\}_\eps\subset\I(p,\alpha)$ given by
\begin{align*}
F_\eps(u;O):=\int_O \f(\frac{x}{\eps},\nabla u(x))dx
\end{align*}
for all $(u,O)\in W^{1,p}(\Omega;\RR^m)\times\O(\Omega)$. Then $\FF(\cdot;O)$ $\Gammap$-converges to 
\begin{align*}
\FF_0(u;O)=\int_O W_{\rm hom}(\nabla u(x))+D_\lambda\Phi_0(x)dx
\end{align*}
for all $u\in W^{1,p}(\Omega;\RR^m)\times\O(\Omega)$, and where $W_{\rm hom}(\xi)$ is given by the formula~\eqref{homog-formula-periodic} with $\S_\xi^W$ in place of $\S_\xi^\f$.
Indeed, \ref{L1} holds because of the $p$-growth of $W$ and~\ref{uniform-bound-l1-example-bis}. Now, we have for almost all $x\in \Omega$
\begin{align*}
\lim_{\rho\to 0}\lim_{\eps\to 0}\fint_{\Q_\rho(x)}\Phi(\frac{y}{\eps})dy=D_\lambda\Phi_0(x)
\end{align*}
since~\eqref{weak-to-measure-example}. Using \cite[Theorem 2.1]{licht-michaille02} we have for every $\xi\in\MM^{m\times d}$
\begin{align*}
W_{\rm hom}(\xi)+D_\lambda\Phi_0(x)=\limsup_{\rho\to 0}\limsup_{t\to\infty}\frac{\S_\xi^\f(t\Q_{\rho}(x))}{\lambda(t\Q_\rho(x))}=\liminf_{\rho\to 0}\liminf_{t\to\infty}\frac{\S_\xi^\f(t\Q_{\rho}(x))}{\lambda(t\Q_\rho(x))}\quad\mbox{a.e. in }\Omega,
\end{align*}
since we can see that for every $x\in\Omega$, every $t\ssup0$ and every $\rho\ssup0$
\begin{align*}
\frac{\S_\xi^\f(t\Q_{\rho}(x))}{\lambda(t\Q_\rho(x))}=&\inf\left\{\fint_{t\Q_\rho(x)} W(y,\xi+\nabla\varphi(y))dy:\varphi\in W^{1,p}_0(t\Q_\rho(x);\RR^m)\right\}\\
&+\fint_{t\Q_\rho(x)}\Phi(y)dy.
\end{align*}
It means that $\f$ is an {\rm H}-integrand and Theorem~\ref{abstract-homogen} apply with
\begin{align*}
\f_{\rm hom}(x,\xi)=W_{\rm hom}(\xi)+D_\lambda\Phi_0(x).
\end{align*}
\end{example}
\begin{remark} An interesting problem in the field of deterministic homogenization (see~\cite{nguetseng-nhang-woukeng10}) is the characterization of all {\rm H}-integrands $\f:\RR^d\times\MM^{m\times d}\to [0,\infty[$ Borel measurable with $p$-growth and $p$-coercivity, i.e., satisfying
\begin{eqnarray*}
&\exists \alpha\ssup0\quad\exists\beta\ssup0\quad\forall (x,\xi)\in\Omega\times\MM^{m\times d}\\ \\
&\alpha\vert \xi\vert^p\le \f(x,\xi)\le\beta(1+\vert\xi\vert^p).
\end{eqnarray*}
\end{remark}
\subsection{Relaxation}\label{subsection-relaxation}
The following result is an extension of Acerbi-Fusco-Dacorogna relaxation theorem (see~\cite[Theorem 9.8, p. 432]{dacorogna08} and~\cite[Statement III.7, p. 144]{acerbi-fusco84}) in the case where the integrand is assumed Borel measurable only.
\begin{theorem}\label{general-relaxation-theorem} If $\f:\Omega\times\RR^m\times\MM^{m\times d}\to [0,\infty[$ is Borel measurable and satisfies~\ref{H0} and~\eqref{p-growth-p-coercivity} then for every $O\in\O(\Omega)$
\begin{align*}
\FF_0(u;O)=\int_O {\f_0}(x,u(x),\nabla u(x))dx
\end{align*}
where for a.a. $x\in O$
\begin{align*}
&\f_0(x,u(x),\nabla u(x))\\
&=\lim_{\rho\to 0}\inf\left\{\fint_{\Q_\rho(x)}\f(y,w(y),\nabla w(y))dy:w\in u_x+W^{1,p}_0(\Q_\rho(x);\RR^m)\right\}.
\end{align*}
Moreover, if $\f$ is Carath\'eodory, i.e.,
\begin{hyp2}
\item\label{cara-meas} for each $(v,\xi)\in\RR^m\times\MM^{m\times d}$ the function $\Omega\ni x\mapsto \f(x,v,\xi)$ is measurable;
\item\label{cara-cont} for a.a. $x\in\Omega$ the function $\RR^m\times\MM^{m\times d}\ni (v,\xi)\mapsto \f(x,v,\xi)$ is continuous,
\end{hyp2}
then for almost every $x\in\Omega$ and for every $(v,\xi)\in\RR^d\times\MM^{m\times d}$
\begin{align}\label{reduc-formula-cara-relax}
{\ftild_0}(x,v,\xi)=\inf\left\{\int_Y \f(x,v,\xi+\nabla\varphi(y))dy:\varphi\in W^{1,\infty}_0(Y;\RR^m)\right\}.
\end{align}
\end{theorem}
\begin{proof} The formula~\eqref{reduc-formula-cara-relax} follows from Proposition~\ref{Reduction formula for Q_pf}. \mbox{\qedhere}
\end{proof}
\begin{remark} Under the same assumptions of Theorem~\ref{general-relaxation-theorem} and using Proposition~\ref{abstract-constant-dirichlet} we also have
\begin{align*}
\FF_0(u;O)=\FF^\D_0(u;O)=\int_O \f_0(x,u(x),\nabla u(x))dx
\end{align*}
for all $(u,O)\in W^{1,p}(\Omega;\RR^m)\times\O(\Omega)$.
\end{remark}

We can give an extension of $W^{1,p}$-quasiconvexity as follows. 
\begin{definition} We say that a Borel measurable integrand $\f:\Omega\times\RR^d\times\MM^{m\times d}\to [0,\infty]$ is $W^{1,p}$-quasiconvex if for every $(x,v,\xi)\in\Omega\times\RR^d\times\MM^{m\times d}$
\begin{align*}
{\ftild_0}(x,v,\xi)=\f(x,v,\xi).
\end{align*}
\end{definition}
However, when the integrand is dependent on $(x,v)$ this generalization of quasiconvexity is more difficult to handle. When the integrand $\f$ is Carath\'eodory the variables $x$ and $v$ can be frozen and we recover the classical concept of quasiconvexity.
\begin{proposition}\label{Reduction formula for Q_pf} If $L$ is Carath\'eodory and satisfies $p$-growth~\eqref{p-growth-p-coercivity} then for a.a. $x\in\Omega$ and for every $(v,\xi)\in\RR^m\times\MM^{m\times d}$ we have
\begin{align}\label{formula-eq-dac}
\ftild_0(x,v,\xi)=\inf\left\{\int_Y \f(x,v,\xi+\nabla\varphi(y))dy:\varphi\in W^{1,\infty}_0(Y;\RR^m)\right\}.
\end{align}
\end{proposition}
\begin{proof}For each $(x,v,\xi)\in\Omega\times\RR^m\times\MM^{m\times d}$ we denote by $Q^{\rm dac}\f(x,v,\xi)$ the right hand side of~\eqref{formula-eq-dac}. For each $\rho\in]0,1[$ we define $\Lambda_\rho, \f_\rho:\Omega\times\RR^m\times\MM^{m\times d}\to [0,\infty[$ by
\begin{align*}
\Lambda_\rho(x,v,\xi):=\inf\bigg\{\int_Y &\f(x+\rho y,v+\rho(\xi y+\psi(y)),\xi+\nabla\psi(y))dy:\\
&\psi\in W^{1,\infty}_0(Y;\RR^m)\bigg\};\\
\f_\rho(x,v,\xi):=\inf\bigg\{\int_Y &\f(x+\rho y,v+\rho(\xi y+\varphi(y)),\xi+\nabla\varphi(y))dy:\\
&\varphi\in W^{1,p}_0(Y;\RR^m)\bigg\}.
\end{align*}
It is easy to see, by a change of variables, that for a.a. $x\in\Omega$ and for every $(v,\xi)\in\RR^m\times\MM^{m\times d}$ we have 
\begin{align}\label{eq-dac0}
\limsup_{\rho\to 0}\f_\rho(x,v,\xi)=\ftild_0(x,v,\xi).
\end{align}
It is enough to show that for a.a. $x\in\Omega$, for every $(v,\xi)\in\RR^m\times\MM^{m\times d}$ and every $\rho\in]0,1[$ it hold
\begin{align}
Q^{\rm dac}\f(x,v,\xi)&=\lim_{\rho\to 0}\Lambda_\rho(x,v,\xi)\label{eq-dac1}\\
\Lambda_\rho(x,v,\xi)&= \f_\rho(x,v,\xi)\label{eq-dac2}.
\end{align}
Indeed, combining~\eqref{eq-dac0},~\eqref{eq-dac1} and~\eqref{eq-dac2} we obtain~\eqref{formula-eq-dac}.

\medskip

{\em Proof of~\eqref{eq-dac1}.} Let $\delta\ssup0$. By Scorza-Dragoni theorem, there exists a compact set $K_\delta\subset \overline{Y}$ such that $\lambda(Y\setminus K_\delta)\sinf\delta$ and $\f\lfloor_{K_\delta\times(\RR^m\times\MM^{m\times d})}$ is continuous. 
Fix $(x,v,\xi)\in\Omega\times\RR^m\times\MM^{m\times d}$ such that
\begin{align}\label{lebesgue-point-a}
a(x)=\lim_{\rho\to 0}\fint_{\Q_\rho(x)} a(y)dy=\lim_{\rho\to 0}\int_{Y} a(x+\rho y)dy\sinf\infty.
\end{align}

We show first that $\limsup_{\rho\to 0}\Lambda_\rho(x,v,\xi)\le Q^{\rm dac}\f(x,v,\xi)$. Note that
\begin{align*}
Q^{\rm dac}\f(x,v,\xi)\le \f(x,v,\xi)\le \beta(a(x)+\vert v\vert^p+\vert \xi\vert^p)\sinf\infty.
\end{align*}
Let $\eps\ssup 0$. There exists $\psi\in W^{1,\infty}_0(Y;\RR^m)$ such that
\begin{align}\label{eps-mini-dac-formula}
\int_Y \f(x,v,\xi+\nabla\psi(y))dy\le\eps+Q^{\rm dac}\f(x,v,\xi).
\end{align}

Fix $\rho\in]0,1[$. Set $g_{\rho}(y):=\f(x+\rho y,v+\rho(\xi y+\psi(y)),\xi+\nabla\psi(y))$ and $g_{0}(y):=\f(x,v,\xi+\nabla\psi(y))$ for all $y\in Y$. Using~\eqref{eps-mini-dac-formula} we have
\begin{align}\label{ineq-formula-dac-relation}
\Lambda_\rho(x,v,\xi)&\le \int_{K_\delta} g_\rho(y)dy+\int_{Y\setminus K_\delta} g_\rho(y)dy\\
&=\int_{K_\delta} g_\rho(y)-g_{0}(y)dy+\int_{Y\setminus K_\delta} g_\rho(y)-g_{0}(y)dy\notag\\
&+\int_Y g_{0}(y)dy\notag\\
&\le \int_{K_\delta} \left\vert g_\rho(y)-g_{0}(y)\right\vert dy+\int_{Y\setminus K_\delta} \left\vert g_\rho(y)-g_{0}(y)\right\vert dy\notag\\
&+\eps+Q^{\rm dac}\f(x,v,\xi).\notag
\end{align}
By using the $p$-growth~\eqref{p-growth-p-coercivity} it easy to see that there exists $C$ depending on $\beta$ and $p$ only such that
\begin{align}\label{domin-diff-dac}
\max\left\{g_0(y),g_\rho(y)\right\}\le C(a(x+\rho y)+\vert v\vert^p+\vert\xi\vert^p+\vert\psi(y)\vert^p+\vert\nabla\psi(y)\vert^p)\mbox{ a.e. in }Y.
\end{align}
By continuity of $\f\lfloor_{K_\delta\times(\RR^m\times\MM^{m\times d})}$ we have $g_\rho(y)-g_{0}(y)\to 0$ a.e. in $K_\delta$ as $\rho\to 0$. Using the domination~\eqref{domin-diff-dac} we obtain by applying the Lebesgue dominated convergence theorem
\begin{align}\label{ineq-formula-dac-relation1}
\lim_{\rho\to 0}\int_{K_\delta} \left\vert g_\rho(y)-g_{0}(y)\right\vert dy=0.
\end{align}
By~\eqref{domin-diff-dac} we have
\begin{align}\label{out-kd}
&\int_{Y\setminus K_\delta} \left\vert g_\rho(y)-g_{0}(y) \right\vert dy\\
&\le 2C(\int_{Y\setminus K_\delta} \!\!\!\!\!\!\!a(x+\rho y)dy+\delta(a(x)+\vert v\vert^p+\vert\xi\vert^p+\Vert\psi\Vert_\infty^p+\Vert\nabla\psi\Vert_\infty^p )).\notag
\end{align}
Note that $\{Y\ni y\mapsto a(x+\rho y)\}_{\rho\in]0,1[}$ is uniformly integrable since~\eqref{lebesgue-point-a}. So, taking the supremum over $\rho$ and passing to the limit $\delta\downarrow 0$ in~\eqref{out-kd} we find that
\begin{align}\label{ineq-formula-dac-relation2}
\lim_{\delta\downarrow0}\sup_{\rho\in]0,1[}\int_{Y\setminus K_\delta} \left\vert g_\rho(y)-g_{0}(y) \right\vert dy=0.
\end{align}
Taking~\eqref{ineq-formula-dac-relation1} and~\eqref{ineq-formula-dac-relation2} into account in~\eqref{ineq-formula-dac-relation} we find
\begin{align*}
\limsup_{\rho\to 0}\Lambda_\rho(x,v,\xi)\le\eps+Q^{\rm dac}\f(x,v,\xi).
\end{align*}

Now, we want to show that $\liminf_{\rho\to 0}\Lambda_\rho(x,v,\xi)\ge Q^{\rm dac}\f(x,v,\xi)$. Consider a sequence $\{\rho_n\}_{n\in\NN}\subset ]0,1[$ such that 
\begin{align*}
\liminf_{\rho\to 0}\Lambda_\rho(x,v,\xi)=\lim_{n\to\infty}\Lambda_{\rho_n}(x,v,\xi)\le  \beta(a(x)+\vert v\vert^p+\vert \xi\vert^p)\sinf\infty
\end{align*}
since $p$-growth conditions~\eqref{p-growth-p-coercivity}. Fix $n\in \NN$. We can choose $\psi_n\in W^{1,\infty}_0(Y;\RR^m)$ such that
\begin{align*}
\int_Y g_n(y)dy\le \rho_n+\Lambda_{\rho_n}(x,v,\xi)
\end{align*}
where $g_{n}(y):=\f(x+\rho_n y,v+\rho_n(\xi y+\psi_{\rho_n}(y)),\xi+\nabla\psi_{\rho_n}(y))$ for all $y\in Y$. Since $p$-coercivity, we can choose a subsequence (not relabelled) such that 
\begin{align}
&\psi_{n}\to\psi_\infty\mbox{ in }L^p(Y;\RR^m)\label{strong-convergence-dac};\\
&\nabla\psi_n\wto\nabla\psi_\infty\mbox{ in }L^p(Y;\MM^{m\times d}).\label{weak-convergence-dac}
\end{align}
Fix $\delta\ssup0$ and choose a compact set $K_\delta\subset \overline{Y}$ such that $\lambda(Y\setminus K_\delta)\sinf\delta$ and $\f\lfloor_{K_\delta\times(\RR^m\times\MM^{m\times d})}$ is continuous. We have by Eisen convergence theorem \cite[p. 75]{eisen79} that
\begin{align}\label{eisen-convergence}
g_n(y)-\f(x,v,\xi+\nabla\psi_n(y))\to 0\mbox{ in measure in }K_\delta.
\end{align}
We have 
\begin{align*}
\int_Y g_n(y)dy\ge &\int_{K_\delta} g_n(y)-\f(x,v,\xi+\nabla\psi_n(y))dy\\
&+\int_{Y\setminus K_\delta} g_n(y)-\f(x,v,\xi+\nabla\psi_n(y))dy+Q^{\rm dac}\f(x,v,\xi).
\end{align*}
Using growth conditions we have for a.a. $y\in Y$
\begin{align}\label{out-kd1}
&\left\vert g_n(y)-\f(x,v,\xi+\nabla\psi_n(y)) \right\vert\\
& \le 2C(a(x+\rho_n y)+a(x)+\vert v\vert^p+\vert\xi\vert^p+\vert\psi_n(y)\vert^p+\vert\nabla\psi_n(y)\vert^p).\notag
\end{align}
By taking~\eqref{eisen-convergence},~\eqref{out-kd1},~\eqref{strong-convergence-dac} and ~\eqref{weak-convergence-dac} into account we have
\begin{align*}
\lim_{n\to\infty} \int_{K_\delta} \left\vert g_n(y)-\f(x,v,\xi+\nabla\psi_n(y))\right\vert dy=0
\end{align*}
since Vitali convergence theorem. Using~\eqref{out-kd1} and reasoning similarly as in the first part of the proof we have
\begin{align*}
\lim_{\delta\to0}\sup_{n\in\NN} \int_{Y\setminus K_\delta} \left\vert g_n(y)-\f(x,v,\xi+\nabla\psi_n(y))\right\vert dy=0.
\end{align*}
It follows that 
\begin{align*}
\liminf_{\rho\to 0} \Lambda_{\rho}(x,v,\xi)=\liminf_{n\to\infty} \Lambda_{\rho_n}(x,v,\xi)\ge \liminf_{n\to\infty}\int_Y g_n(y)dy\ge Q^{\rm dac}\f(x,v,\xi).
\end{align*}

\medskip

{\em Proof of~\eqref{eq-dac2}.} Fix $(x,v,\xi)\in\Omega\times\RR^m\times\MM^{m\times d}$ and $\rho\in ]0,1[$. We only need to prove that
\begin{align}\label{eq-dac2-real}
\f_\rho(x,v,\xi)\ge \Lambda_{\rho}(x,v,\xi).
\end{align}
Let $\eps\ssup0$. There exists $\varphi_\eps\in W^{1,p}_0(Y;\RR^m)$ such that
\begin{align*}
\f_\rho(x,v,\xi)+\eps\ge \int_Y \f(x+\rho y,v+\rho(\xi y+\varphi_{\eps}(y)),\xi+\nabla\varphi_{\eps}(y))dy.
\end{align*}
There exists a sequence $\{\psi_n\}_{n\in\NN}\subset W^{1,\infty}_0(Y;\RR^m)$ such that $\psi_n\to \varphi_\eps$ in $W^{1,p}(Y;\RR^m)$, $\psi_n\to \varphi_\eps$ a.e. in $Y$ and $\nabla\psi_n\to \nabla\varphi_\eps$ a.e. in $Y$ as $n\to\infty$. Using growth conditions we have for some $C$ depending on $\beta$ and $p$ only, for a.a. $y\in Y$ and for all $n\in\NN$
\begin{align*}
&\f(x+\rho y,v+\rho(\xi y+\psi_{n}(y)),\xi+\nabla\psi_{n}(y))\\
&\le C(a(x+\rho y)+\vert v\vert^p+\vert\xi\vert^p+\vert\psi_n(y)\vert^p+\vert\nabla\psi_n(y)\vert^p).
\end{align*}
Since $\f$ is Carath\'eodory we have
\begin{align*}
\lim_{n\to\infty}&\f(x+\rho y,v+\rho(\xi y+\psi_{n}(y)),\xi+\nabla\psi_{n}(y))\\&=\f(x+\rho y,v+\rho(\xi y+\varphi_{\eps}(y)),\xi+\nabla\varphi_{\eps}(y))\quad\mbox{ a.e. in }Y.
\end{align*}
Applying Vitali convergence theorem we obtain
\begin{align*}
\Lambda_\rho(x,v,\xi)&\le \lim_{n\to\infty}\int_Y\f(x+\rho y,v+\rho(\xi y+\psi_{n}(y)),\xi+\nabla\psi_{n}(y))dy\\
&=\int_Y \f(x+\rho y,v+\rho(\xi y+\varphi_{\eps}(y)),\xi+\nabla\varphi_{\eps}(y))dy\\
&\le \f_\rho(x,v,\xi)+\eps.
\end{align*}
Letting $\eps\to 0$ we finally obtain~\eqref{eq-dac2-real}.\mbox{\qedhere}
\end{proof}
%%%%%%%%%%%%%%APPENDIX%%%%%%%%%%%%%%%%%
\section{Appendix}
\subsection{Usage of Vitali covering theorem} Let $A\subset O\in\O(\Omega)$ be a set which is not necessarily measurable. For each $x\in A$ we consider a family of closed balls $\K_x$ containing $x$ of $O$ satisfying $\inf\left\{\diam(\Q):\Q\in {\K_x}\right\}=0$ and $A\subset\mathop{\cup}_{\Q\in\K}\Q$ with $\K:=\mathop{\cup}_{x\in A}\K_x$. We say that $\K$ is a fine cover of $A$.

Then there exists a countable pairwise disjointed family of balls $\left\{\overline{\Q}_i\right\}_{i\ge 1}\subset {\K}$ such that
\begin{align*}
\lambda(A\setminus {\mathop{\cup}_{i=1}^{\infty}{{\Q}_i}})=0.
\end{align*}
It follows that for any $\mu\in \mesabs(O)$, i.e. $\mu\ll\lambda{\lfloor_O}$, we have $\mu(A\setminus {\mathop{\cup}_{i\ge 1}{{\Q}_i}})=0$. Moreover, if $\lambda(A)\sinf\infty$ then for any $\delta\ssup0$ we can choose a finite subfamily $\left\{\overline{\Q}_i\right\}_{i= 1}^{N}\subset {\K}$ satisfying
\begin{align*}
\mu(A\setminus {\mathop{\cup}_{i= 1}^N{{\Q}_i}})\sinf\delta.
\end{align*}

\subsection{Level sets of derivative of set functions} Let $G:\Bal_o(\Omega)\to ]-\infty,\infty]$ be a set function. Let $O\in\O(\Omega)$. For each $h\in\RR$ we consider the strict sublevel (resp. superlevel) of the lower (resp. upper) derivative of $G$
\begin{align*}
S_h:=\left\{x\in O:\Di_\lambda G(x)\sinf h\right\}\quad (\mbox{resp. }S^h:=\left\{x\in O:\Ds_\lambda G(x)\ssup h\right\})
\end{align*}

The following lemma give consequences of sublevel (resp. superlevel) sets of derivative of set functions. 
\begin{lemma}\label{consequenceSh-0} Let $h\in\RR$ and $\eta\ssup0$. Then 
\begin{hyp2}
\item\label{consequenceSh-1} there exists a {\em countable} pairwise disjointed family $\{\Q_i\}_{i\in I}\subset\Bal_o(O)$ such that
\begin{align}\label{consequenceSh}
\lambda(S_h\setminus \mathop{\cup}_{i\in I}\Q_i)=0,\quad \forall i\in I\quad G(\Q_i)\sinf h\lambda(\Q_i)\mbox{ and }\;\diam(\Q_i)\in ]0,\eta[
\end{align}
(resp. $\displaystyle
\lambda(S^h\setminus  \mathop{\cup}_{i\in I}\Q_i)=0,\quad \forall i\in I\quad G(\Q_i)\ssup h\lambda(\Q_i)\mbox{ and }\;\diam(\Q_i)\in ]0,\eta[
$);\\

\item\label{consequenceSh-2} for every $\delta\ssup0$ there exists a {\em finite} pairwise disjointed family $\{\Q_i\}_{i\in I}\subset\Bal_o(O)$ such that
\begin{align*}
\lambda(S_h\setminus  \mathop{\cup}_{i\in I}\Q_i)\sinf\delta,\quad \forall i\in I\quad G(\Q_i)\sinf h\lambda(\Q_i)\mbox{ and }\;\diam(\Q_i)\in ]0,\eta[
\end{align*}
(resp. $\displaystyle
\lambda(S^h\setminus  \mathop{\cup}_{i\in I}\Q_i)\sinf\delta,\quad \forall i\in I\quad G(\Q_i)\ssup h\lambda(\Q_i)\mbox{ and }\;\diam(\Q_i)\in ]0,\eta[
$).
\end{hyp2}
\end{lemma}
\begin{proof} Let $h\in\RR$ and $\eta\ssup0$. We only give the proof for $S_h$, since similar arguments apply for $S^h$. Note that~\ref{consequenceSh-2} is a direct consequence of~\ref{consequenceSh-1}, so, we only show~\ref{consequenceSh-1}. 

If $x\in S_h$ then for some $\eps\ssup0$
\begin{align*}
\forall \rho\in ]0,\eta[\qquad\inf\left\{\frac{G(\Q)}{\lambda(\Q)}: \Q \in \mathfrak{B}_{x,\rho}(O)\right\}\sinf h-\eps
\end{align*}
where $\mathfrak{B}_{x,\rho}(O):=\left\{\Q:x\in\Q\in\Bal_o(O)\mbox{ and }\diam(\Q)\le\rho\right\}$. For each $\rho\in ]0,\eta[$ there exists $\Q_{x,\rho}\in \mathfrak{B}_{x,\rho}(O)$ such that
\begin{align}\label{equation-practical-vitali}
\frac{G(\Q_{x,\rho})}{\lambda(\Q_{x,\rho})}-\eps\le \inf\left\{\frac{G(\Q)}{\lambda(\Q)}: \Q \in \mathfrak{B}_{x,\rho}(O)\right\}\sinf h-\eps.
\end{align}
Consider the family $\K_\eta:=\left\{\overline{\Q_{x,\rho}}\right\}_{x\in S_h,\rho\in ]0,\eta[}$ of closed cubes such that~\eqref{equation-practical-vitali} holds. The family $\K_\eta$ is a fine cover of $S_h$, i.e.,  
\begin{align*}
S_h\subset \mathop{\cup}_{\Q\in\K_\eta}\Q\quad\mbox{ and }\quad\forall x\in S_h\quad\inf\left\{\diam(\Q):\Q\in \K_{\eta,x}\right\}=0
\end{align*}
where $\K_{\eta,x}:=\left\{\overline{\Q_{x,\rho}}\right\}_{\rho\in ]0,\eta[}\subset\K_\eta$. By Vitali covering theorem we conclude~\eqref{consequenceSh}.\mbox{\qedhere}
\end{proof}

\subsection{Proof of Lemma~\ref{measurability}}\label{proof-measurability} Fix $c\in \RR$. We have to prove that 
\begin{align*}
M_c:=\left\{x\in O: \Di_\lambda G(x)\le c\right\}
\end{align*}
is measurable. Fix $\eta\ssup0$. Set $h:=c+\eta$. By Lemma~\ref{consequenceSh-0}~\ref{consequenceSh-1} there exists a countable pairwise disjointed family $\{\Q_i\}_{i\in I}\subset\Bal_o(O)$ such that
\begin{align*}
\lambda(S_h\setminus  \mathop{\cup}_{i\in I}\Q_i)=0,\quad \forall i\in I\quad G(\Q_i)\sinf h\lambda(\Q_i)\mbox{ and }\;\diam(\Q_i)\in ]0,\eta[.
\end{align*}
Since $S_h\supset M_c$ we have
\begin{align*}
\lambda(M_c\setminus  \mathop{\cup}_{i\in I}\Q_i)=0.
\end{align*}
If we show that the Borel set $\Q^\infty:=\mathop{\cup}_{i\in I}{\Q}_i\subset M_c$ then $M_c$ will be the reunion of a Borel set and a $\lambda$-negligible set and so measurable since $\lambda$ is complete. Let $z\in \Q^\infty$. Then there exists $i_z\in I$ such that $z\in {\Q}_{i_z}$. It follows that 
\begin{align*}
\inf\left\{\frac{G(\Q)}{\lambda(\Q)}:z\in\Q\in\Bal_o(\Omega),\;\diam(\Q)\le \eta\right\}\le \frac{G(\Q_{i_z})}{\lambda(\Q_{i_z})}\le  c+\eta.
\end{align*}
Passing to the limit $\eta\to0$ we obtain $\Di_\lambda G(z)\le c$ which means that $z\in M_c$. The proof is complete.\quad\mbox{\qedsymbol}

\subsection{Properties of the family of set functions $\{\M_\eps(u;\cdot)\}_\eps$}
\begin{lemma}\label{lemma-sub-sup-m} Let $(u,O)\in W^{1,p}(\Omega;\RR^m)\times \O(\Omega)$. Then the family $\{\M_\eps(u;\cdot)\}_\eps,\;\M_\eps(u;\cdot):\O(O)\to [0,\infty]$ satisfies
\begin{hyp2}
%\item\label{domination} there exists $\mu_{\G}\in \mesabs(O)$ such that for every $U\in\O(O)$ \[\sup\limits_{\delta>0}G_\delta(U)\le \mu_\G(U);\]

%\medskip

\item\label{additive} for every $\eps\ssup0$ and every $(U,V)\in\O(O)\times\O(O)$ 
\[
U\cap V=\emptyset\implies \M_\eps(u;U\cup V)\le \M_\eps(u;U)+\M_\eps(u;V);
\]

\medskip

\item\label{implication-abs} for every $\eps\ssup 0$, every $U,V\in\O(O)$ with $U\subset V$ 
\[
\lambda(V\ssetminus U)=0\implies \M_\eps(u;U)=\M_\eps(u;V);
\]
%\end{hyp1}

%\begin{hyp1}[resume]
\item\label{additivity-m}in particular, for every $U\in\O(O)$ and $V\in\O(O)$ satisfying $U\subset V$ we have for every $\eps\ssup0$
\[%\label{additivity-m}
\lambda(\partial U)=0\implies \M_\eps(u;V)\le\M_\eps(u;U)+\M_\eps(u;V\setminus \overline{U}).
\]
\end{hyp2}
\end{lemma}
\begin{proof} We recall that for $A\in\O(\Omega)$ we have
\begin{align*}
W^{1,p}_0(A;\RR^m)=\left\{u\in W^{1,p}(\Omega;\RR^m):u=0\mbox{ in }\Omega\setminus A\right\}.
\end{align*}

If $U,V\in\O(O)$ satisfy $U\cap V=\emptyset$ then for every $\varphi_i\in L^p(\Omega;\RR^m)$ with $i\in\{0,1,2\}$ we have
\begin{align*}
%&\varphi_0\in W^{1,p}_0(U\cup V;\RR^m)\Longrightarrow \varphi_0\mathds{1}_U\in W^{1,p}_0(U;\RR^m)\mbox{ and }\varphi_0\mathds{1}_V\in W^{1,p}_0(V;\RR^m);\\
&\varphi_1\in W^{1,p}_0(U;\RR^m)\mbox{ and }\varphi_2\in W^{1,p}_0(V;\RR^m)\Longrightarrow \varphi_1\mathds{1}_U+\varphi_2\mathds{1}_V\in W^{1,p}_0(U\cup V;\RR^m).
\end{align*}

\medskip

Let $\eps\ssup0$. To verify~\ref{additive} it suffices to write for every $\varphi_1\in W^{1,p}_0(U;\RR^m)\mbox{ and }\varphi_2\in W^{1,p}_0(V;\RR^m)$
\begin{align*}
F_\eps(u+\varphi_1;U)+F_\eps(u+\varphi_2;V)&= F_\eps(u+\varphi_1\mathds{1}_U+\varphi_2\mathds{1}_V;U\cup V)\\
 &\ge \M_\eps(u;U\cup V),
\end{align*}
taking the infimum over $\varphi_1$ and $\varphi_2$ we obtain
\begin{align*}
\M_\eps(u;U)+\M_\eps(u;V)\ge \M_\eps(u;U\cup V).
\end{align*}

\medskip

Consider $U,V\in\O(O)$ satisfying $U\subset V$ and $\lambda(V\setminus U)=0$. Since $U\subset V$ we have $W^{1,p}_0(U;\RR^m)\subset W^{1,p}_0(V;\RR^m)$, thus $\M_\eps(u;U)\ge \M_\eps(u;V)$. Assume that $\M_\eps(u;V)<\infty$. For every $\eta\ssup0$ there exists $\varphi\in W^{1,p}_0(V;\RR^m)$ such that $\infty\ssup\M_\eps(u;V)+\eta\ge F_\eps(u+\varphi;V)$. By using \ref{C1} we have
\begin{align*}
\M_\eps(u;V)+\eta&\ge F_\eps(u+\varphi;V)=F_\eps(u+\varphi\mathds{1}_U;U)+F_\eps(u+\varphi;V\setminus U)\\
&\ge \M_\eps(u;U).
\end{align*}
Note that $\varphi\mathds{1}_U=\varphi$ a.e. in $V$ and so $\varphi\mathds{1}_U\in W^{1,p}_0(U;\RR^m)$. Therefore~\ref{implication-abs} is satisfied. 

\medskip

To prove~\ref{additivity-m} it is sufficient to use the properties~\ref{implication-abs},~\ref{additive} together with the fact that we can write
$V\setminus(U\cup(V\setminus\overline{U}))=\partial U$ for all $U,V\in\O(O)$ satisfying $U\subset V$.\mbox{\qedhere}\end{proof}


\begin{thebibliography}{DMM86b}

\bibitem[AF84]{acerbi-fusco84}
Emilio Acerbi and Nicola Fusco.
\newblock Semicontinuity problems in the calculus of variations.
\newblock {\em Arch. Rational Mech. Anal.}, 86(2):125--145, 1984.

\bibitem[AH96]{adams-hedberg}
David~R. Adams and Lars~Inge Hedberg.
\newblock {\em Function spaces and potential theory}, volume 314 of {\em
  Grundlehren der Mathematischen Wissenschaften [Fundamental Principles of
  Mathematical Sciences]}.
\newblock Springer-Verlag, Berlin, 1996.

\bibitem[BB00]{bellieud-bouchitte00}
Guy Bouchitt{\'e} and Michel Bellieud.
\newblock Regularization of a set function--application to integral
  representation.
\newblock {\em Ricerche Mat.}, 49(suppl.):79--93, 2000.
\newblock Contributions in honor of the memory of Ennio De Giorgi (Italian).

\bibitem[BFLM02]{BFLM02}
Guy Bouchitt{\'e}, Irene Fonseca, Giovanni Leoni, and Lu{\'{\i}}sa Mascarenhas.
\newblock A global method for relaxation in {$W^{1,p}$} and in {${\rm SBV}_p$}.
\newblock {\em Arch. Ration. Mech. Anal.}, 165(3):187--242, 2002.

\bibitem[BFM98]{bouchitte-fonseca-mascarenhas98}
Guy Bouchitt{\'e}, Irene Fonseca, and Luisa Mascarenhas.
\newblock A global method for relaxation.
\newblock {\em Arch. Rational Mech. Anal.}, 145(1):51--98, 1998.

\bibitem[Bon82]{bongiorno81}
Benedetto Bongiorno.
\newblock Essential variations.
\newblock In {\em Measure theory, {O}berwolfach 1981 ({O}berwolfach, 1981)},
  volume 945 of {\em Lecture Notes in Math.}, pages 187--193. Springer, Berlin,
  1982.

\bibitem[Dac08]{dacorogna08}
Bernard Dacorogna.
\newblock {\em Direct methods in the calculus of variations}, volume~78 of {\em
  Applied Mathematical Sciences}.
\newblock Springer, New York, second edition, 2008.

\bibitem[DG79]{degiorgi79}
Ennio De~Giorgi.
\newblock Convergence problems for functionals and operators.
\newblock In {\em Proceedings of the {I}nternational {M}eeting on {R}ecent
  {M}ethods in {N}onlinear {A}nalysis ({R}ome, 1978)}, pages 131--188, Bologna,
  1979. Pitagora.

\bibitem[DMM86a]{dalmaso-modica86}
Gianni Dal~Maso and Luciano Modica.
\newblock Integral functionals determined by their minima.
\newblock {\em Rend. Sem. Mat. Univ. Padova}, 76:255--267, 1986.

\bibitem[DMM86b]{dalmaso-modica86-2}
Gianni Dal~Maso and Luciano Modica.
\newblock Nonlinear stochastic homogenization.
\newblock {\em Ann. Mat. Pura Appl. (4)}, 144:347--389, 1986.

\bibitem[DMM86c]{dalmaso-modica86-3}
Gianni Dal~Maso and Luciano Modica.
\newblock Nonlinear stochastic homogenization and ergodic theory.
\newblock {\em J. Reine Angew. Math.}, 368:28--42, 1986.

\bibitem[Eis79]{eisen79}
Goswin Eisen.
\newblock A selection lemma for sequences of measurable sets, and lower
  semicontinuity of multiple integrals.
\newblock {\em Manuscripta Math.}, 27(1):73--79, 1979.

\bibitem[LM02]{licht-michaille02}
Christian Licht and G{\'e}rard Michaille.
\newblock Global-local subadditive ergodic theorems and application to
  homogenization in elasticity.
\newblock {\em Ann. Math. Blaise Pascal}, 9(1):21--62, 2002.

\bibitem[Mas06]{dalmaso2006}
G.~Dal Maso.
\newblock {$\Gamma$}-convergence and homogenization.
\newblock In Editors in~Chief: Jean-Pierre~Fran\c coise, Gregory~L. Naber, and Tsou~Sheung Tsun, editors, {\em Encyclopedia of Mathematical Physics},
  pages 449 -- 457. Academic Press, Oxford, 2006.

\bibitem[Mod86]{modica86}
L.~Modica.
\newblock Stochastic homogenization and ergodic theory.
\newblock In {\em Optimization and related fields ({E}rice, 1984)}, volume 1190
  of {\em Lecture Notes in Math.}, pages 359--370. Springer, Berlin, 1986.

\bibitem[NNW10]{nguetseng-nhang-woukeng10}
Gabriel Nguetseng, Hubert Nnang, and Jean~Louis Woukeng.
\newblock Deterministic homogenization of integral functionals with convex
  integrands.
\newblock {\em NoDEA Nonlinear Differential Equations Appl.}, 17(6):757--781,
  2010.

%\bibitem[WS69]{wright-snyderI}
%Harvel Wright and W.~S. Snyder.
%\newblock On the differentiability of arbitrary real-valued set functions. {I}.
%\newblock {\em Trans. Amer. Math. Soc.}, 145:439--454, 1969.

\end{thebibliography}
\end{document}